\documentclass[10pt]{amsart}
\usepackage{amssymb,amsmath}
\usepackage{hyperref}
\usepackage{bm}
\usepackage[]{graphicx}
\usepackage{amssymb, epsfig}

\usepackage{color}

\usepackage{tikz}
\usetikzlibrary{shapes.geometric,arrows,positioning}

\numberwithin{equation}{section}

\numberwithin{figure}{section}
\newcommand{\eps}{\varepsilon}

\newcommand{\CM}{{\mathcal M}}

\newcommand{\CO}{{\mathcal O}}

\newtheorem{theorem}{Theorem}[section]
\newtheorem{lemma}[theorem]{Lemma}
\newtheorem{definition}[theorem]{Definition}

\newtheorem{remark}[theorem]{Remark}
\newtheorem{corollary}[theorem]{Corollary}
\begin{document}
%
%
%
%
%
\title[Motion of a droplet for the mass-conserving Stochastic Allen-Cahn
equation]{Motion of a droplet for the mass-conserving\\
Stochastic Allen-Cahn equation}
\author[Antonopoulou, Bates, Bl\"omker, Karali]{
D.C.~Antonopoulou$^{\$*}$, P.W.~Bates$^{\ddag}$, D.
Bl\"omker$^{\sharp}$, G.D.~Karali$^{\dag*}$ }

%
%
%
%
%
%
\thanks
{P.B. was supported in part by NSF DMS 0908348. P.B. and G.K.
were supported in part by the IMA at the University of Minnesota, where
this work was completed. D.A. and G.K. were supported by
`Aristeia' (Excellence) grant, 193, ${\rm \Sigma\Pi A}$ 00086.}
\thanks
{$^{\ddag}$ Department of Mathematics, Michigan State University,
East Lansing, MI 48824, USA}
\thanks{$^{\sharp}$ Institut f\"ur Mathematik, Universit\"at Augsburg, 86135 Augsburg, Germany}
\thanks
{$^{\dag}$ Department of Mathematics and Applied Mathematics,
University of Crete, GR--714 09 Heraklion, Greece.}
\thanks
{$^{\$}$ Department of Mathematics, University of Chester,
Thornton Science Park, CH2 4NU, UK}
\thanks
{$^{*}$ Institute of Applied and Computational Mathematics,
FORTH, GR--711 10 Heraklion, Greece.}
%

%
%

%
%
\begin{abstract}
 We study the stochastic mass-conserving Allen-Cahn equation
posed on a bounded domain of $\mathbb{R}^2$ with additive spatially smooth
space-time noise. This equation associated with a small
positive parameter $\eps$ describes the stochastic motion of a
small almost semicircular droplet attached to domain's boundary
and moving towards a point of locally maximum curvature. We apply
It\^o calculus to derive the stochastic dynamics of the droplet by
utilizing the approximately invariant manifold introduced by
Alikakos, Chen and Fusco \cite{acfu} for the deterministic problem.
In the stochastic case depending on the scaling,
the motion is driven by the change in the curvature of the boundary
and the stochastic forcing. Moreover, under the assumption of a sufficiently
small noise strength, we establish stochastic stability of a neighborhood
of the
manifold of droplets in $L^2$ and $H^1$,
which means that with overwhelming probability the solution
stays close to the manifold for very long time-scales.
%
\end{abstract}

\maketitle \textbf{Keywords:} {\small{Stochastic Allen-Cahn, mass
conservation, droplet's motion, additive noise, invariant
manifold, stochastic dynamics, stochastic stability, It\^o
calculus.}}
%
%
\pagestyle{myheadings}
\thispagestyle{plain}
%
%
%
\section{Introduction}
\subsection{The problem}
We consider the IBVP for the mass conserving Allen-Cahn equation
posed on a two-dimensional bounded smooth domain
$\Omega$ and introduce an additive spatially smooth and white in time space-time
noise $\dot{V}$ 
\begin{equation}\label{eqphi}
\begin{split}
&\partial_t \phi^{\hat{\eps}}(y,t)=\hat{\eps}^2\Delta_y\phi^{\hat{\eps}}(y,t)-f(\phi^{\hat{\eps}}(y,t))
+\frac{1}{|\Omega|}\int_{\Omega}f(\phi^{\hat{\eps}}(y,t))dy+\dot{V}(y,t),\;\;\;y\in\Omega,\;\;\;t>0,\\
&\partial_n\phi^{\hat{\eps}}(y,t)=0,\;\;\;y\in\partial\Omega,\;\;\;t>0,\\
&\phi^{\hat{\eps}}(y,0)=\phi_0^{\hat{\eps}}(y),\;\;\;y\in\Omega.
\end{split}
\end{equation}
Here, $\hat{\eps}$ is a small positive parameter, $\Omega\subset \mathbb{R}^2$ of area
$|\Omega|$ is a bounded domain with sufficiently smooth boundary $\partial
\Omega$, and $\partial_n$ is the exterior normal derivative to
$\partial \Omega$.  The function $f$ is
the  derivative of a double-well potential, which we denote by
$F$. We assume that $f$ is smooth,
$f(\pm 1)=0< f'(\pm1)$ and $f$ has exactly one other zero
that lies in $(-1,1)$. The standard example is  $f(u)=u^3-u$,
which we will assume for simplicity in the whole presentation,
although the result holds for more general nonlinearities.

The deterministic problem, i.e., when $\dot{V}=0$, was first
studied by Rubinstein and Sternberg, \cite{rsten}, then by
Alikakos, Chen and Fusco, \cite{acfu}, and later by Bates and Jin
in \cite{BJ}. In \cite{acfu}, the authors analyzed the problem's
long-time dynamics and established existence of stable sets of
solutions corresponding to the motion of a small, almost
semicircular interface (droplet) intersecting the boundary of the
domain and moving towards a point of locally maximal curvature.
In \cite{BJ}, the authors established the existence of a global
invariant manifold of droplet states using the approximation
given in \cite{acfu}.

The Allen-Cahn equation, also called Model A in the theory of dynamics
of critical phenomena (cf.\ \cite{hohenberg}), describes the evolution of the
concentration $\phi^{\hat\eps}$ of one species of a two-phase mixture, for example a binary alloy,
occupying a
bounded domain $\Omega$. The small positive parameter
$\hat\eps$ represents the surface tension associated with
interfacial regions that are generated during phase separation, cf.\ \cite{AC}. The
double-well potential $F$ favors layered functions that take
values close to its minima $\pm 1$. The zero level sets of such a
function are called interfaces and the values close to $\pm 1$ are called
states. Usually they are
assigned almost uniformly away from the interface .

Due to mass conservation, a phase separation begins either by
spinodal decomposition, or as in our case, as the mass is very
asymmetric  by nucleation. For the case of Cahn-Hilliard equation
see \cite{Wa1, Wa2}.

If the states are separated, the total perimeter of
interfaces decreases in time, \cite{fife2,fife3,B}. For the
one-dimensional case see also \cite{CGS,M,MM,S,OS}. For the
two-dimensional single layer problem, Chen and Kowalczyk in
\cite{CK} proved that in the limit $\hat\eps\rightarrow 0^+$ this
layer becomes a circular arc interface intersecting the boundary
orthogonally and encloses a point on the boundary where the
curvature has a local maximum. Alikakos, Chen and Fusco in
\cite{acfu} restricted the analysis to a single connected
interface (curve) of shape close to a small semicircular-arc
intersecting the outer boundary $\partial\Omega$. This so called droplet
maintains a semicircular shape for economizing the perimeter
and therefore, in \cite{acfu} its evolution was fully described in
terms of the motion of its center along the outer boundary.

In the absence of the non-local term, the multi-dimensional
stochastic Allen-Cahn equation driven by a multiplicative noise
non-smooth in time and smooth in space was considered in
\cite{rogweb}; the authors therein prove the tightness of
solutions for the sharp interface limit problem. We refer also to
the results in \cite{hiweb}, where a mollified additive white
space-time noise was introduced and the limiting behavior of
solution was investigated for very rough noise. Considering the
one-dimensional Allen-Cahn equation with an additive space-time
white noise, in \cite{weber2}, the author proved exponential
convergence towards a curve of minimizers of the energy.

In this paper, as in \cite{acfu}, we consider a single small
droplet and so, the average concentration $m\in(-1,1)$ is assumed
to satisfy
\[
m=1-\frac{\pi\delta^2}{|\Omega|},
\]
for some $0<\delta\ll1$ while the parameter $\hat{\eps}$ is
sufficiently small such that $0<\hat{\eps}\ll\delta^3$.

When $\dot{V}:=0$, if $z(\hat{\xi}_0)$ is a point of
$\partial\Omega$ where the curvature has a strict extremum, then
there exists a unique equilibrium $\phi(y)$ of \eqref{eqphi} with
zero level set close to the circle of radius $\delta$ centered at
this point. Moreover, for layered initial data whose interface is
close to the semicircle centered at $z(\hat{\xi}_0)$ of radius
$\delta$, the solution of \eqref{eqphi} is layered also, with
interface close to a semicircle of the same radius centered at
some point $z(\hat{\xi})$ of $\partial\Omega$, \cite{acfu}.

The mass conservation constraint
$$\frac{1}{|\Omega|}\int_\Omega\phi^{\hat{\eps}}(y,t)dy=m\;\;\;\mbox{for any}\;\;\;t\geq 0$$
holds if and only if
\begin{equation*}
\begin{split}
0&=\partial_t\Big{(}\int_\Omega\phi^{\hat{\eps}}(y,t)dy\Big{)}
=\int_\Omega\phi_t^{\hat{\eps}}(y,t)dy\\
&=\int_\Omega\hat{\eps}^2\Delta_y\phi^{\hat{\eps}}(y,t)dy
-\int_\Omega f(\phi^{\hat{\eps}}(y,t))dy
+\int_\Omega\frac{1}{|\Omega|}\int_{\Omega}f(\phi^{\hat{\eps}}(y,t))dydy
+\int_\Omega\dot{V}(y,t)dy\\
&=\int_\Omega\dot{V}(y,t)dy.
\end{split}
\end{equation*}
For the  calculation above we integrated \eqref{eqphi} in space
and used the Neumann boundary conditions. Hence, mass conservation
holds for the stochastic problem only if the spatially smooth
additive noise satisfies
\begin{equation}\label{massconstoc}
\int_\Omega\dot{V}(y,t)dy=0\;\;\;\mbox{for
any}\;\;\;t\geq 0.
\end{equation}
This means that in a Fourier series expansion, there is no noise
on the constant mode. Otherwise the average mass would behave
like a Brownian motion, and would not stay close to $1$.

Following \cite{acfu}, in order to fix the size of the droplet, we
introduce in \eqref{eqphi} the following change of variables:
\begin{equation}\label{rescal}
y=\delta
x,\;\;\;\hat{\eps}=\eps\delta,\;\;\;u^{\eps}(x,t)=\phi^{\hat{\eps}}(y,t),\;\;\;
\dot{W}(x,t)=\dot{V}(y,t),\;\;\;
\Omega_\delta=\delta^{-1}\Omega:=\{x\in \mathbb{R}^2:\;\delta
x\in\Omega\},
\end{equation}
and obtain the equivalent problem
\begin{equation}\label{equ}
\begin{split}
&\partial_t u^{\eps}(x,t)=\eps^2\Delta u ^{\eps}(x,t)-f(u^{\eps}(x,t))
+\frac{1}{|\Omega_\delta|}\int_{\Omega_\delta}f(u^{\eps}(x,t))dx
+\dot{W}(x,t),\;\;\;x\in\Omega_\delta,\;\;\;t>0,\\
&\partial_n u^{\eps}(x,t)=0,\;\;\;x\in\partial\Omega_\delta,\;\;\;t>0,\\
&u^{\eps}(x,0)=u_0^{\eps}(x),\;\;\;x\in\Omega_\delta.
\end{split}
\end{equation}
Here $\Delta=\Delta_x$, $\dot{W}(x,t)$ is again an additive
smooth in space, space-time noise defined
below and $\partial_n$
is the normal derivative to $\partial\Omega_\delta$.


\subsection{Assumptions on the noise}
The noise $\dot{W}$ is defined as the formal derivative of a
Wiener process depending on $\eps$, which is given by a Fourier
series with coefficients being independent Brownian motions in
time. Since $\dot{W}$ arises from a rescaling of the noise
$\dot{V}$, we also could take care of the dependence on $\delta$,
but here we suppose that $\delta>0$ is small but fixed (see
Remark \ref{rem:noise}).

Let $W$ be a $\mathcal{Q}$-Wiener process in the underlying
Hilbert space $H:=L^2(\Omega_\delta)$, where $\mathcal{Q}$
is a symmetric operator and $(e_k)_{k\in \mathbb{N}}$ is a
complete $L^2(\Omega_\delta)$-orthonormal basis of eigenfunctions with corresponding
eigenvalues $a_k^2$, so that
\[
 \mathcal{Q} e_k=a_k^2e_k.
\]
Then $W$ is given as the Fourier series
\[
W(t):=\sum_{k=1}^{\infty}a_k\beta_k(t)e_k(\cdot),\leqno{{\rm(N1)}}
\]
for a sequence of independent real-valued Brownian motions
$\{\beta_k(t)\}_{t\geq 0}$, cf.\ DaPrato and Zabzcyck
\cite{DPZ92}. Note that, due to rescaling, $a_k$, $\mathcal{Q}$,
and $e_k$ will depend on $\delta$. We suppress this dependence in
our notation.

The process $W$ is assumed to satisfy
$$
\int_{\Omega_\delta}\dot{W}(x,t)dx=0\;\;\;\mbox{for
any}\;\;\;t\geq 0,\leqno{{\rm(N2)}}
$$
so that the mass conservation condition \eqref{massconstoc} holds
true.

As our approach is  based 
 on application of
It${\rm \hat{o}}$-formula, we will always assume that the trace
of the operator $\mathcal{Q}$ is finite, i.e.,\
\[
 {\rm trace} (\mathcal{Q}):=\sum_{k=1}^\infty a_k^2=\eta_0<\infty.
\]
Furthermore, let $\|\mathcal{Q}\|$ be the induced $L^2$ operator
norm, then the noise strength is defined by
$$\|\mathcal{Q}\|=\eta_1.\leqno{{\rm(N3)}}
$$
Here, observe that
$$\eta_1=\|\mathcal{Q}\|\leq {\rm trace}(\mathcal{Q})=\eta_0.$$
The required smoothness in
space of the noise is given by
$$\eta_2=
\sum_{i=1}^{\infty}a_i^2\|\nabla e_i\|^2=\text{trace}(Q\Delta)<\infty\;.\leqno{{\rm(N4)}}
$$
This assumption will be used in the sequel, when the It\^o-formula
will be applied for the proof of certain $H^1$-norm estimates.

Our results will depend on the size of $(\eta_0, \eta_1, \eta_2)$ in terms of $\eps$.
The usual scenario would be that all $\eta_i$ have a common prefactor in $\eps$, which is the noise-strength,
and are otherwise independent of $\eps$.

\begin{remark}
\label{rem:noise}
Note that due to the rescaling, if we assume that $V$ did not
depend on $\delta$, then it is $\eta$ that depends on $\delta$.
More specifically, since we are in two dimensions then $e_k(x) =
{\delta}^{-1} f_k(\delta x)$, where $f_k$ is an ONB in
$L^2(\Omega)$. Thus $V=\delta W$, and all $a_k$ are of order
$\CO(\delta)$. Hence, $\eta_0$ and $\eta_1$ are of order
$\CO(\delta^2)$, while $\eta_2$ is of order $\CO(\delta^4)$.

Our philosophy in this paper will be to consider $\delta$ very
small but fixed, and analyze the asymptotic problem for
$0<\eps\ll1$. Thus, we suppress the explicit dependence on
$\delta$ in the notation.
\end{remark}


\subsection{The droplet}

We define, for a smooth
function $v$, the operator
$$
\mathcal{L}^{\eps}(v):=\eps^2\Delta
v-f(v)+\frac{1}{|\Omega_\delta|}\int_{\Omega_\delta}f(v)dx\;\;\;\mbox{in}\;\;\;\Omega_\delta,\;\;\;\partial_n
v=0\;\;\;\mbox{on}\;\;\;\partial\Omega_\delta,
$$
and fix the cubic nonlinearity $f$ as in the introduction.

Following Theorem 2.5 of \cite{acfu} (p. 267), we have:
\begin{lemma}
\label{lem:defu}
For any integer $K\in \mathbb{N}$ and for $\delta$, $\eps$
sufficiently small parameters satisfying
%
\begin{equation}\label{upbound}
\eps\leq \frac{1}{2} C_1^*\delta^2,
\end{equation}
with $$C_1^*:=\frac{8\pi
f'(1)}{3\sqrt{6}|\Omega|\int_{-1}^1\sqrt{F(s)}ds},$$ there exist
a droplet like state $u=u(x,\xi,\eps)$ and a scalar (velocity)
field $c=c(\xi,\eps)$ such that
\begin{equation}\label{a4.1}
\begin{split}
&\mathcal{L}^\eps(u)=\eps^2c(\xi,\eps)\partial_\xi u+\mathcal{O}_{L^\infty}(\eps^K)\;\;\;\mbox{in}\;\;\;\Omega_\delta,\\
&\partial_nu=0\;\;\;\mbox{on}\;\;\;\partial\Omega_\delta,\\
&\int_{\Omega_\delta}u=|\Omega_\delta|-\pi,
\end{split}
\end{equation}
where the scalar $\xi\in\left(0,|\partial\Omega_\delta|\right)$
is the arc-length parameter of $\partial\Omega_\delta$.
\end{lemma}

Here, $\mathcal{O}_{L^\infty}(\eps^K)$ denotes a term that is
uniformly bounded by $C\eps^K$ for some constant $C>0$.

\begin{figure}[ht]
 \begin{center}
    \begin{tikzpicture}[scale=2,>=latex]
    \draw [black,line width=1.3pt] plot [smooth] coordinates {(-1,1) (1,-.2) (2.5,-.5) (4,-.2) (7,1)};
     \draw [red,line width=3pt] plot [smooth] coordinates { (2,-.47)  (2.05,-0.24) (2.2,-0.1) (2.5,-0.03) (2.8,-.1) (2.95,-0.24)  (3,-.47)};
      \draw [red] (2.7,0.1) node {\LARGE  $\Gamma$};
     \draw [black] (6,0.35) node {\LARGE  $\partial\Omega_\delta$};%
     \draw [black] (1.5,0.5) node {\LARGE  $\Omega_\delta$};
        \filldraw[black] (2.5,-.5) circle(1pt);
         \draw [black] (2.6,-.7) node {\Large  $\xi$};
         \draw [black] (2.6,-.2) node {$1$};
          \draw[black] (2.5,-.5)--(2.8,-.1);
    \end{tikzpicture}
    \end{center}

\caption{Droplet state in the rescaled domain $\Omega_\delta$ with semicircular arc $\Gamma$}
\end{figure}
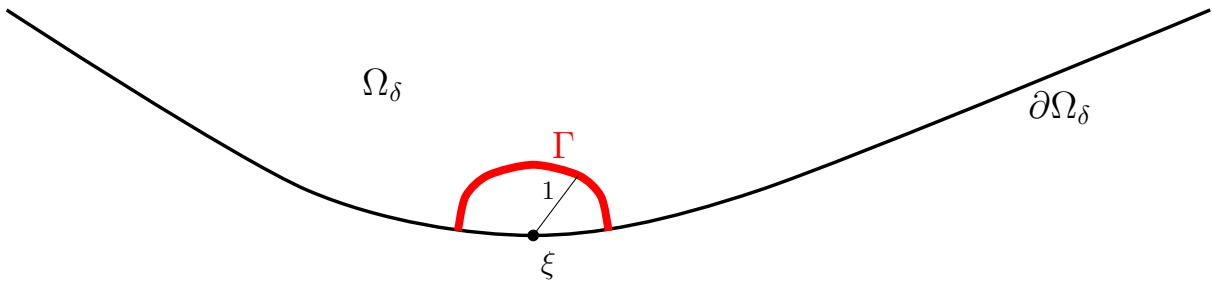

\begin{proof}[Idea of Proof]
The droplet is constructed, using asymptotic expansions in $\eps$,
as a function of $(r,s)$, with $r$ the signed distance from the
interface $\Gamma$ and $s$ the arc-length along $\Gamma$, this
being an approximately semicircular curve intersecting
$\partial\Omega_\delta$ orthogonally. Additional asymptotic
expansions are used near the corners where the interface meets
$\partial \Omega_\delta$ but these are of higher order. The first
order approximation to this state is $U(r/\eps)$, transverse to
$\Gamma$ at each point, where $U$ is the solution to
\begin{equation}\label{heterocl}
\ddot{U}-f(U)=0,\;\;\;U(\pm\infty)=\pm
1,\quad \text{such that }\int_\mathbb{R}R\dot{U}^2(R)dR=0.
\end{equation}
\end{proof}

 \begin{figure}[h]
 \begin{center}
    \begin{tikzpicture}[scale=2,>=latex]
        \draw[black,arrows=->,line width=1.3pt] (0,0)--(3.5,0);
        \draw[black,arrows=->,line width=1.3pt] (0,-.3)--(0,1.3);
         \draw [black] (0,-.5) node {\Large $\partial\Omega_\delta$};
         \draw [black] (3.8,0) node {\Large $\Omega_\delta$};
          \draw [black] (-.2,1) node {$1$};
         \draw [red,line width=1.3pt] plot [smooth] coordinates {(0,1) (0.5,0.99) (.8,0.9) (.9,0.6) (1,0) (1.1,-0.6) (1.2,-0.75) (1.5,-0.88)(2,-0.95) (3,-0.99) (4,-1)};
        \draw [red] (1.25,.5) node {$U(r/\eps)$};
         \draw[black, <->] (0.8,-.1)--(1.2,-.1);
         \draw [black] (0.8,-.3) node {$\eps$};

    \end{tikzpicture}
    \end{center}
  \caption{Sketch of a section through the Droplet showing the local shape given by $U$.}
 \end{figure}
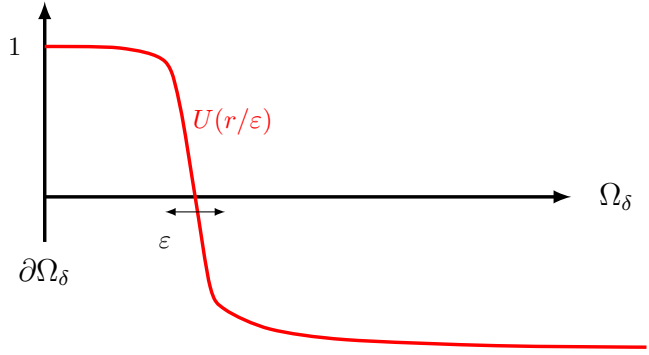

In previous lemma, and for the rest of this paper $\partial_\xi
u:=\frac{\partial}{\partial\xi} (u(\cdot))$, while
$\partial_\xi^l(u):=\frac{\partial^l}{\partial\xi^l}(u(\cdot))$
for any integer $l>1$. Note that $\partial_\xi u$ is just the
usual partial derivative when $u$ is considered as a function of
$x$, parameterized by $\xi$ and $\eps$, but when $u$ is
considered as a function of local coordinates $(r,s)$, then one
must take into account the fact that $r$ and $s$ both vary with
$\xi$.  This observation will be used in the final section where
these derivatives are estimated.

According to \cite{acfu} (p. 261 and relation (2.54) on p. 262) it
follows that:

\begin{lemma}
\label{lem:shapec} For $c$ from the definition of the droplet in
Lemma \ref{lem:defu}, we have
\begin{equation}\label{cest}
c=c(\xi,\eps)=\mathcal{O}(\delta^2).
\end{equation}
 Moreover,
 \[
c(\xi,\eps)=
-\frac{4}{3\pi}g_0\mathcal{K}_{\Omega_\delta}^{\prime}(\xi)
\delta\Big{[}1+\mathcal{O}(\delta)\Big{]}+\mathcal{O}(\delta^4)
\]
where $\mathcal{K}_{\Omega_\delta}^{\prime}(\xi)$ is the
derivative of the curvature of $\partial\Omega_\delta$, and
$g_0$ is a constant equal to $1$ if $\eps=\mathcal{O}(\delta^3)$.
\end{lemma}


\subsection{The manifold}

We define a manifold of droplet states, which approximates
solutions well and captures the motion of droplets along the
manifold.

\begin{definition}
We consider the manifold
$$\mathcal{M}:=\Big{\{}u(\cdot,\xi,\eps):\;\xi\in [0,|\partial\Omega_\delta|]\Big\},$$ consisting of the smooth
functions $u$, mentioned above, having a droplet-like structure and satisfying
\eqref{a4.1}.
\end{definition}

Obviously, $\mathcal{M}$ is a closed manifold without boundary
which in homotopy equivalent to the boundary
$\partial\Omega_\delta$. Furthermore, if $\Omega_\delta$ is
simply connected, then $\mathcal{M}$ is topologically a circle.

We define some tubular neighborhoods of the manifold
$\mathcal{M}$ in which we shall work. For $r>0$, let
\begin{equation}\label{nn}
\begin{split}
&\mathcal{N}_{L^2}^r:=\Big{\{} v \in
L^2(\Omega_\delta):\;d_{L^2}(v,\mathcal{M}) < r \Big{\}},\\
&\mathcal{N}_{H_{\eps}^1}^r:=\Big{\{} v \in
H^1(\Omega_\delta):\;d_{H_\eps^1}(v,\mathcal{M}) < r \Big{\}},
\end{split}
\end{equation}
where $d_{L^2}$ is the distance in the $L^2(\Omega_\delta)$-norm
and $d_{H_\eps^1}$ the distance in the norm
$\|\cdot\|_{H_\eps^1}:H^1(\Omega_\delta)\rightarrow \mathbb{R}^+$
defined by
$$\|v\|_{H_\eps^1}:=\Big{(}\eps^2\|\nabla v \|_{L^2(\Omega_\delta)}^2+
\|v\|_{L^2(\Omega_\delta)}^2\Big{)}^{\frac{1}{2}}.$$ The usual
Sobolev space $H^1(\Omega_\delta)$ equipped with the norm
$\|\cdot\|_{H_\eps^1}$ will be denoted by
$H_\eps^1:=\Big{\{}H^1(\Omega_\delta),\;\|\cdot\|_{H_\eps^1}\Big{\}}$.
For the rest of this paper $(\cdot,\cdot)$ will denote
the $L^2(\Omega_\delta)$-inner product and $\|\cdot\|$ the
induced $L^2(\Omega_\delta)$-norm.

Let us finally remark, that for $r$ sufficiently small (depending
on $\delta$) there is a well defined local coordinate system in
both $\mathcal{N}_{L^2}^r$ and $\mathcal{N}_{H_{\eps}^1}^r$ and
the projection onto $\CM$ is well defined and smooth.


\subsection{Main Results}

In Section 2, we analyze the dynamics of solutions $w:=u^\eps$ of
\eqref{equ} approximated by some $u$ in $\mathcal{M}$ and written
as
\begin{equation*}
w(t)=u(\cdot,\xi(t),\eps)+v(t).
\end{equation*}
Here, $v$ is orthogonal on the manifold, i.e.,\
$v(t)\bot_{L^2(\Omega_\delta)}
\partial_\xi(u(\cdot,\xi(t),\eps))$.

We suppose that $\|v\|$ is small and
$\xi$ is a diffusion process given by
\begin{equation*}
d\xi=b(\xi)dt+(\sigma(\xi),dW),
\end{equation*}
for some scalar field $b:\mathbb{R}\rightarrow \mathbb{R}$ and
some variance $\sigma:\mathbb{R}\rightarrow H$.
In Remark \ref{rem:just} we shall comment on the fact that it is
not restrictive to assume $\xi$ being a diffusion.

Applying It\^o
calculus, we compute first $b$ and $\sigma$ exactly and then
estimate their size in terms of $\eps$, in order to determine the
major contribution. See Corollary \ref{cor:sde} and Remarks
\ref{rem:strato} and \ref{rem:curv}.

Under the assumption of a sufficiently smooth initial condition,
by Theorem \ref{stdyn} we prove
that locally in time
\begin{equation*}
\begin{split}
d\xi&\approx \eps^2c(\xi(t),\eps)dt
+d\mathcal{A}_s,
\end{split}
\end{equation*}
where the stochastic process $\mathcal{A}_t$ is given as a
diffusion process by the formula
\begin{equation}\label{def:dAt}
d\mathcal{A}_s=A^{-1}\Big{[}\frac{1}{2}(v,\partial_\xi^3
(u))-\frac{3}{2}(\partial_\xi^2 u,\partial_\xi u)\Big{]}
(\mathcal{Q}\sigma,\sigma)dt+A^{-1}(\sigma,\mathcal{Q}
\partial_\xi^2 u)dt+A^{-1}(\partial_\xi u,dW),
\end{equation}
for $A:=\|\partial_\xi u\|^2-(v,\partial_\xi^2 u)$.

Theorem \ref{neffect} and its corollary estimates the effect of noise on the local
in time stochastic dynamics driven by the additive term
$d\mathcal{A}_s$ which supplies the deterministic dynamics of
\cite{acfu} with an extra deterministic drift and a noise term.
We will see (cf.\ Remark \ref{rem:strato})
that this is for small $v$ the Wiener process $W$ projected
to the manifold given by a Stratonovic differential
\[
d\mathcal{A}_s = A^{-1}(\partial_\xi u,\circ dW)\;.
\]
Since the deterministic dynamics are of  order
$\mathcal{O}(\eps^2)$ then the noise is not always dominant.
\textit{Only, if the noise strength is sufficiently large, or the
curvature of the boundary is constant, then noise can dominate
(cf.~Remark \ref{rem:curv})}. Also if the droplet sits in a spot with maximal curvature,
then the noise dominates, at least locally around that point.

This is in contrast to the one-dimensional
stochastic Cahn-Hilliard equation (see \cite{abk}) where it has
been proved that a noise of polynomial strength in $\eps$ cannot
be ignored since the deterministic dynamics are exponentially
small.


In Section 3, we give sufficient criteria for the noise strength
such that the solution stays with high probability close to the
droplet states both in $L^2$ and $H^1_\eps$ norms, until times of
any polynomial order $\mathcal{O}(\eps^{-q})$. The previous is
achieved by estimating the $p$-moments of $v$ in various norms.

In our $L^2$-stability result of Theorem \ref{newl2thm} under the
assumption of  $\eta_0=\mathcal{O}(\eps^{2k-2-\tilde{k}})$ for
some small $\tilde{k}>0$, we prove that if $w(0)$ lies in
$\mathcal {N}_{L^2}^{\eta\eps^{k-2}}$ for $k>5$ and some $\eta>0$
independent of $\eps$, then with high probability the solution
stays in a slightly larger neighborhood $\mathcal
{N}_{L^2}^{C\eps^{k-2}}$ ($C>\eta$) for any long time of order
$\mathcal{O}(\eps^{-q})$, $q\in\mathbb{N}$.

Our problem is stochastic, so stability in $H^1_\eps$-norm is
investigated analytically since we can not refer to the abstract
parabolic regularity argument used in \cite{acfu} in the absence
of noise. More specifically, in Theorem \ref{newh1thm} and its
corollary, imposing the additional assumption of
$\eta_2=\mathcal{O}(\eps^{2k-6})$, we show that if $w(0)$ lies in
$\mathcal {N}_{L^2}^{\eta\eps^{k-2}}$ and $\nabla w(0)$ lies in
$\mathcal {N}_{L^2}^{\eta\eps^{k-4}}$ then with high probability
the solution  stays in $\mathcal
{N}_{H^1_\eps}^{C\eps^{k-3-\tilde{k}}}$ for any long time of order
$\mathcal{O}(\eps^{-q})$, $q\in\mathbb{N}$, where  $\tilde{k}$ is
just some arbitrary small number. \textit{Due to this result, the
local in time stochastic dynamics derived in Section 2, are
proven to be valid for very long time scales}. Nevertheless, we
can not claim that the radius of stability is the optimal one.

The last independent Section \ref{sec44} involves some higher
order estimates needed for the stochastic dynamics. We compute
these estimates by extending the analogous lower order results of
\cite{acfu} which were used for the deterministic problem.

Throughout this manuscript, as many of our proofs are quite
technical, we present the application of It\^o calculus in full
details; we hope that the interested reader may gain a wider
comprehension of this stochastic technique. Finally, let us remark,
that we denote various constant all by $C$,
although their value may change from line to line.


\section{Stochastic dynamics}\label{sec2}


\subsection{The exact stochastic equation of droplet's motion}

In this section, we shall derive the stochastic motion on the
manifold $\mathcal{M}$ following the main lines of Theorem 4.3
\cite{acfu} (p. 297), presented for the deterministic problem. For
simplicity we use the symbol $w$ in place of $u^{\eps}$, so that
problem \eqref{equ} can be written as the
following stochastic PDE
\begin{equation}\label{oper}
dw =\mathcal{L}^{\eps}(w) dt+ d{W}.
\end{equation}

Let the position on $\partial\Omega$ be a diffusion process  $\xi$ given by
\begin{equation}\label{xi}
d\xi=b(\xi)dt+(\sigma(\xi),dW),
\end{equation}
for some scalar field $b:\mathbb{R}\rightarrow \mathbb{R}$ and
some variance $\sigma:\mathbb{R}\rightarrow H$ still to be
determined. We will justify this ansatz later in Remark
\ref{rem:just}, once we obtain equations for $b$ and $\sigma$.

We approximate the solution $w$ of \eqref{equ} by some
$u=u(\xi(t))$ in $\mathcal{M}$, and write
\begin{equation}\label{1}
w(t)=u(\cdot,\xi(t),\eps)+v(t),
\end{equation}
where
$v\bot_{L^2(\Omega_\delta)}
\partial_\xi u$. We use equation
\eqref{equ} written as \eqref{oper} to get
\begin{equation}\label{2}
d w= \Big[\eps^2\Delta w-f(w)
+\frac{1}{|\Omega_\delta|}\int_{\Omega_\delta}f(w)dx\Big]dt+dW
=\left[\mathcal{L}^\eps(u)-Lv-N(u,v)\right] dt+dW,
\end{equation}
where $L$ is defined, for $v$ smooth, by
$$Lv=-\eps^2\Delta v+f'(u)v,$$
so that $-L$ is the linearization of $\mathcal{L}^\eps$ at $u$,
and $N(u,v)$ is the remaining nonlinear part. Differentiating $w=u+v$ with respect to
$t$, we obtain by It${\rm \hat{o}}$ calculus
\begin{equation}\label{3}
dw=\partial_\xi u\;d\xi+dv+\frac{1}{2}\partial_\xi^2u\;d\xi d\xi.
\end{equation}
From It\^o-calculus  $dtdt=dW dt=0$ and higher order
differentials all vanish.

Moreover, for the quantity $d\xi d \xi$ the following lemma holds
true.
\begin{lemma}
\label{lem:help1} We have that
\begin{equation}\label{help1}
d\xi d\xi =(\mathcal{Q}\sigma(\xi),\sigma(\xi))dt.
\end{equation}
\end{lemma}
\begin{proof}
 Note that using It\^o-calculus, by the definition of $\xi$ we derive
$$
d\xi d\xi=(\sigma(\xi),dW)(\sigma(\xi),dW).
$$
Thus, the claim  follows immediately from the definition of the
covariance operator.

In more detail, we use the series expansion of $W$ together with
$d\beta_i d\beta_j=\delta_{ij}dt$, and derive from Parcevals identity
for arbitrary
functions $a$ and $b$ the relation
\begin{equation}\label{hhh}
(a,dW)(b,dW)
=\sum_{i,j}a_ia_jd\beta_i
d\beta_j(a,e_i)(b,e_j)=\sum_ia_i^2(a,e_i)(b,e_i)dt=(\mathcal{Q}
a,b)dt\;.
\end{equation}
\end{proof}

Therefore, substituting \eqref{3} in \eqref{2} we obtain by using
Lemma \ref{lem:help1} above,
\begin{equation}\label{4**}
\begin{split}
dw=\partial_\xi u\;d\xi
+\frac{1}{2}\partial_\xi^2 u\cdot(\mathcal{Q}\sigma,\sigma)dt
+dv
=\Big{[}\mathcal{L}^\eps(u)-Lv-N(u,v)\Big{]}dt+dW.
\end{split}
\end{equation}
We take the $L^2(\Omega_\delta)$-inner product of \eqref{4**}
with $\partial_\xi u$ and arrive at
\begin{equation}\label{5}
\begin{split}
\|\partial_\xi u\|^2d\xi+(dv,\partial_\xi u)=&\Big{[}(\mathcal{L}^\eps(u),\partial_\xi u)-(Lv,\partial_\xi u)-(N(u,v),\partial_\xi u)\Big{]}dt\\
&+(dW,\partial_\xi u)-\frac{1}{2}(\partial_\xi^2 u,\partial_\xi u) \cdot (\mathcal{Q}\sigma,\sigma )dt.
\end{split}
\end{equation}
We differentiate in $t$ the orthogonality condition
$(v,\partial_\xi u)=0$ and obtain by applying again It\^{o}
calculus
\begin{equation}\label{6**}
(dv,\partial_\xi u)+(v,d\partial_\xi u)
+(dv,d\partial_\xi u)=0.
\end{equation}
As it is demonstrated by the following lemma, by making use of the
relation above we shall eliminate $(dv,\partial_\xi u)$ in
\eqref{5}.
\begin{lemma}\label{lem1}It holds that
\begin{equation}\label{7}
\begin{split}
(dv, d[ \partial_\xi u])=&-(v,\partial_\xi^2 u)d\xi-(\sigma,\mathcal{Q} \partial_\xi^2 u)dt\\
&-\frac{1}{2}(v,\partial_\xi^3 u)\cdot (\mathcal{Q}\sigma,\sigma)dt
+(\partial_\xi u,\partial_\xi^2 u)\cdot (\mathcal{Q}\sigma ,\sigma)dt.
\end{split}
\end{equation}
\end{lemma}
\begin{proof}
It\^{o} calculus (recall $d\xi dt =0$ and $dtdt=0$) and (\ref{help1}) yields
$$
d[\partial_\xi u]
=\partial_\xi^2 ud\xi
+\frac{1}{2}\partial_\xi^3 u\cdot (\mathcal{Q}\sigma,\sigma)dt\;.
$$
Then, using this in \eqref{6**}, we obtain after some computations
\begin{equation}\label{6*}
\begin{split}
(\partial_\xi u,dv)&= - (v,d[\partial_\xi u])-(dv,d\partial_\xi u)\\
&=-(v,\partial_\xi^2 u)d\xi
-\frac{1}{2}(v,\partial_\xi^3 u)\cdot(\mathcal{Q}\sigma,\sigma)dt
-(\partial_\xi^2 u,dv)d\xi.
\end{split}
\end{equation}
Observing that $w=u+v$ and hence, $dv=dw-du$, by \eqref{4**}, we
arrive at
\begin{equation}
\begin{split}
-(\partial_\xi^2 u,dv)d\xi=&-(\partial_\xi^2 u,dw)d\xi+(\partial_\xi^2 u,du)d\xi\\
=&-(\partial_\xi^2 u,dW)d\xi+(\partial_\xi^2 u,du)d\xi\\
=&-(\partial_\xi^2 u,dW)d\xi
+(\partial_\xi u ,\partial_\xi^2 u) \cdot(\mathcal{Q}\sigma,\sigma)dt .
\end{split}
\end{equation}
Using our definition of $\xi$ as a diffusion process from
\eqref{xi}, i.e.,\ $d\xi=b(\xi)dt+(\sigma(\xi),dW),$ and relation
\eqref{hhh} we obtain
$$(\partial_\xi^2 u,dW)d\xi
=(\partial_\xi^2 u,dW) (\sigma, dW)
=(\sigma,\mathcal{Q} \partial_\xi^2 u)dt\;.
$$
This yields
\begin{equation}
\label{help4}
-(\partial_\xi^2 u,dv)d\xi=
-(\sigma,\mathcal{Q} \partial_\xi^2 u)dt
+(\partial_\xi u ,\partial_\xi^2u) (\sigma,\mathcal{Q} \sigma )dt\;.
\end{equation}
So, by replacing \eqref{help4} in \eqref{6*} the result follows.
\end{proof}
Now we proceed to derive the equations of motion along the manifold.
Using \eqref{7} in \eqref{5} we arrive at
\begin{equation}
\label{motion*}
\begin{split}
\Big{[}\|\partial_\xi u\|^2-(v,\partial_\xi^2 u)\Big{]}d\xi=&\Big{[}(\mathcal{L}^\eps(u),\partial_\xi u)-(Lv,\partial_\xi u)-(N(u,v),\partial_\xi u)\Big{]}dt\\
&+\Big{[}\frac{1}{2}(v,\partial_\xi^3 u)-\frac{3}{2}(\partial_\xi^2 u,\partial_\xi u)\Big{]}(\mathcal{Q}\sigma,\sigma)dt
+(\sigma,\mathcal{Q}\partial_\xi^2 u)dt
+(\partial_\xi u,dW).
\end{split}
\end{equation}
Obviously, if $A:=\|\partial_\xi u\|^2-(v,\partial_\xi^2 u)\neq
0$, then \eqref{motion*} finally yields:

\subsubsection*{The  stochastic o.d.e. for the droplet's dynamics}

\begin{equation}\label{motion**}
\begin{split}
d\xi=A^{-1}\Big{[}(\mathcal{L}^\eps(u),\partial_\xi u)&-(Lv,\partial_\xi u)-(N(u,v),\partial_\xi u)\Big{]}dt\\
&+A^{-1}\Big{[}\frac{1}{2}(v,\partial_\xi^3u)
-\frac{3}{2}(\partial_\xi^2 u,\partial_\xi u)\Big{]}
(\mathcal{Q}\sigma,\sigma)dt
+A^{-1}(\sigma,\mathcal{Q}\partial_\xi^2 u)dt\\
&+A^{-1}(\partial_\xi u,dW).
\end{split}
\end{equation}
Note that provided $v$ is sufficiently small, the invertibility of $A$  is obvious.
The detailed statement is proven in Lemma \ref{leminv} presented below.

So, since $d\xi=bdt+(\sigma,dW),$ then collecting the `$dt$'
terms we obtain the \textit{formula for the drift $b$}
\begin{equation}\label{fieldb}
\begin{split}
b=A^{-1}\Big{[}(\mathcal{L}^\eps(u),\partial_\xi u)&-(Lv,\partial_\xi u)-(N(u,v),\partial_\xi u)\Big{]}\\
&+A^{-1}\Big{[}\frac{1}{2}(v,\partial_\xi^3u)-\frac{3}{2}(\partial_\xi^2 u,\partial_\xi u)
\Big{]}(\mathcal{Q}\sigma,\sigma)
+A^{-1}(\sigma,\mathcal{Q} \partial_\xi^2 u),
\end{split}
\end{equation}
while the \textit{variance $\sigma$} is given by
\begin{equation}\label{var}
\sigma=A^{-1}\partial_\xi u.
\end{equation}

Finally, as promised, we prove the following lemma which
establishes the invertibility of $A$ and the asymptotic behavior
of $A^{-1}$ as $\eps\rightarrow 0^{+}$.
\begin{lemma}\label{leminv}
For $k>5/2$ and a fixed constant $c>0$, if $\|v\|<c\eps^{k-2}$,
then there exists a constant $C_0>0$ such that $A\geq C_0/{\eps}$
and therefore $A^{-1}$ exists and
\begin{equation}\label{am}
|A^{-1}|=\mathcal{O}(\eps)
\quad\text{as}\quad
\eps\to 0^+.
\end{equation}
\end{lemma}
\begin{proof}
From \cite{acfu} (p. 297) we have the following estimates
\begin{equation}\label{u1}
\|\partial_\xi u\|\geq C^{-1}\eps^{-1/2}
\qquad \text{and} \qquad
\|\partial_\xi^2 u\|\leq C \eps^{-3/2},
\end{equation}
so,
\begin{equation}\label{u3}
\|\partial_\xi^2 u\|\leq C\eps^{-1}\eps^{-1/2}\leq
C^{3}\|\partial_\xi u\|^2\eps^{-1/2}.
\end{equation}
 Easily, \eqref{u1}, \eqref{u3} yield
\begin{equation}\label{u4}
\begin{split}
|A|\geq A=\|\partial_\xi u\|^2-(v,\partial_\xi^2 u)
&\geq\|\partial_\xi u\|^2-\|v\|\|\partial_\xi^2 u\|\geq
\|\partial_\xi u\|^2- C\eps^{k-2}\|\partial_\xi^2 u\|\\
&\geq \|\partial_\xi u\|^2-\tilde C\eps^{k-2}\eps^{-1/2}\|\partial_\xi u\|^2\\
&\geq C\eps^{-1}\Big{[}1+\mathcal{O}(\eps^{k-5/2})\Big{]}\;.
\end{split}
\end{equation}
Hence, provided $k>5/2$, we obtain
$$|A^{-1}|\leq
C\eps\Big{[}1+\mathcal{O}(\eps^{k-5/2})\Big{]}^{-1}=\mathcal{O}(\eps).$$
\end{proof}

\begin{remark}
\label{rem:just} We note that the assumption of $\xi$ being a
diffusion process is not very restrictive. Following the steps of
the derivation backwards, it can be established that for any pair
$(\xi,v)$ where $\xi$ solves (\ref{motion**}) and $v$ solves
\[dv = d[\partial_\xi u]+ \mathcal{L}^\eps(v+u)+f(v+u)dt + dW,
\]
then $v\perp \partial_\xi u$ in $L^2$ and the function $w=u+v$
solves the mass conserving Allen-Cahn equation with noise. See
also the analytical analogous results for the stochastic
Cahn-Hilliard equation in \cite{abk}, or the detailed discussion
in \cite{Weber:thesis} for Allen-Cahn equation without mass-conservation.
\end{remark}


\subsection{The approximate stochastic o.d.e. for the droplet's motion}
Now we proceed by proving the main theorems of Section 2 that
analyze the droplet's exact dynamics and approximations thereof,
in terms of $\eps$.%

Here, we need to assume bounds on the $H^1$-norm of $v$, as we are
not able to bound the nonlinearity otherwise, since the
$L^2$-bound of $v$ cannot control the cubic nonlinearity.

\begin{theorem}
\label{stdyn} For some  $k> 5/2$ and fixed small
$\tilde\kappa>0$, suppose that
$w(0)\in\mathcal{N}_{L^2}^{\eta\eps^{k-2}}$ and that $w(0)
\in\mathcal{N}_{H_\eps^1}^{\eta\eps^{k-3-\tilde\kappa}}$ for any
fixed  constant $\eta$. Then, locally in time (as long as
$w(t)\in\mathcal{N}_{H_\eps^1}^{C\eps^{k-3-\tilde\kappa}} \cap
\mathcal{N}_{L^2}^{C\eps^{k-2}}$ for any fixed large constant
$C>\eta$), it holds that
\begin{equation}\label{adyn}
\begin{split}
d\xi=\eps^2c(\xi(t),\eps)dt +\mathcal{O}(\eps^{k-1/2}\delta^2)dt
+\mathcal{O}(\eps^{\min\{2k-11/2-2\tilde\kappa,3k-17/2-2\tilde\kappa,k-3/2\}})
dt +d\mathcal{A}_t,
\end{split}
\end{equation}
where the stochastic process $\mathcal{A}_t$ defined in (\ref{def:dAt})
is the part in the equation for $\xi$, which arises due to the presence of noise.
\end{theorem}

Recall Lemma \ref{lem:shapec}, where $c=\mathcal{O}(\delta^2)$.
Note that we need $k\geq15/4$ to have the error term in \eqref{adyn} above of order $\mathcal{O}(\eps^2)$.

Using the results of  p. 297 of \cite{acfu}, we obtain:
\begin{lemma}\label{lem:Bbou} Under the assumptions of Theorem
\ref{stdyn}, 
it holds that
 $$
\Big{|}(\mathcal{L}^\eps(u),\partial_\xi u)-(Lv,\partial_\xi
(u))-(N(u,v),\partial_\xi u)\Big{|}=
\Big{[}\eps^2c+\mathcal{O}(\eps^{2k-11/2-2\tilde\kappa})
+\mathcal{O}(\eps^{3k-17/2-2\tilde\kappa})
+\mathcal{O}(\eps^{k-3/2})\Big{]}\|\partial_\xi
u\|^2.$$
\end{lemma}
\begin{proof}
Observe first that by definition \eqref{a4.1} and for some large
$K>0$, we get
\[
(\mathcal{L}^\eps(u),\partial_\xi u) = (c \eps^2\|\partial_\xi
u\|^2+ \CO(\eps^{K})\| \partial_\xi u\|_{L^1} )\;.
\]
Secondly, we will use the following interpolations
\[
\|v\|_{L^4}^2 \leq \|v\|_{H^1}\|v\| \qquad \text{and}\qquad \|v\|_{L^6}^3 \leq \|v\|^2_{H^1}\|v\|
\;.
\]
Since $f$ is  cubic  and $u$ is uniformly bounded, we arrive at
\[| (N(u,v),\partial_\xi u) |
\leq C (\|v\|_{H^1}+\|v\|^2_{H^1} )\|v\|  \| \partial_\xi u\|
\leq C( \eps^{2k-6-2\tilde\kappa} +
\eps^{3k-9-2\tilde\kappa}) \eps^{1/2} \|
\partial_\xi u\|^2\;,
\]
where we used that $$\|v\|_{H^1}\leq \eps^{-1} \|v\|_{H^1_\eps}.$$
Note that this is the only argument in this proof where the
$H^1$-norm appears.

For the third term, using that $\|\partial_\xi u\|\geq
C\eps^{-1/2}$ and $\|\Delta\partial_\xi
u\|\sim\|u_\xi^{(3)}\|\leq C\eps^{-5/2}$ (cf.\ Appendix) we obtain
\begin{equation}
\label{sep1}
\begin{split}
|(Lv,\partial_\xi u)| &= |(v, L [\partial_\xi u])| =
|-\eps^2(v,\Delta \partial_\xi u)+(v,f'(u)\partial_\xi u)|\\
&\leq \eps^2\|v\|\eps^{-5/2}+C\|v\|\|\partial_\xi u\|
\\
&\leq C\eps^{k-2-1/2}+C\eps^{k-2}\|\partial_\xi u\|
\\
&\leq C\eps^{k-3/2}\|\partial_\xi u\|^2.
\end{split}
\end{equation}
\end{proof}

\begin{proof}[Proof of Theorem \ref{stdyn}]
Setting
$$
B:=\Big{[}\frac{1}{2}(v,\partial_\xi^3 u)-\frac{3}{2}(\partial_\xi^2 u,\partial_\xi u)\Big{]},
$$
we get by \eqref{motion*}
\begin{equation}\label{motion***}
\begin{split}
Ad\xi=\Big{[}(\mathcal{L}^\eps(u),\partial_\xi u)&-(Lv,\partial_\xi u)-(N(u,v),\partial_\xi u)\Big{]}dt\\
&+B(\mathcal{Q}\sigma,\sigma)dt
+(\sigma,\mathcal{Q}\partial_\xi^2 u)dt
+(\partial_\xi u,dW).
\end{split}
\end{equation}
Lemma \ref{lem:Bbou} yields

\begin{equation*}
\begin{split}
Ad\xi=&
\Big{[}\eps^2c(\xi,\eps)+\mathcal{O}(\eps^{2k-11/2-2\tilde\kappa})
+\mathcal{O}(\eps^{3k-17/2-2\tilde\kappa})
+\mathcal{O}(\eps^{k-3/2})\Big{]}\|\partial_\xi u\|^2dt
\\
&+B
(\mathcal{Q}\sigma,\sigma)dt
+(\sigma,\mathcal{Q}
\partial_\xi^2 u)dt
+(\partial_\xi u,dW).
\end{split}
\end{equation*}
Therefore, we get by \eqref{def:dAt}
\begin{equation}\label{help9}
d\xi= A^{-1}\Big{[}\eps^2c(\xi,\eps)
+\mathcal{O}(\eps^{2k-11/2-2\tilde\kappa})
+\mathcal{O}(\eps^{3k-17/2-2\tilde\kappa})
+\mathcal{O}(\eps^{k-3/2})\Big{]}\|\partial_\xi (u)\|^2dt +
d\mathcal{A}_s.
\end{equation}
By Lemma \ref{leminv} we have locally in time, as long as
$w(t)\in
\mathcal{N}_{L^2}^{\eta\eps^{k-2}}$,
\begin{equation}\label{help11}
|A^{-1}|\|\partial_\xi u\|^2\leq
\Big{[}1+C\eps^{k-5/2}\Big{]}^{-1}=1+\mathcal{O}(\eps^{k-5/2}).\end{equation}
So, from \eqref{help9} analogously to the arguments in p. 297 of
\cite{acfu}, we have
\begin{equation}\label{help10}
\begin{split}
d\xi&=
\Big{[}\eps^2c(\xi,\eps)+\mathcal{O}(\eps^{2k-11/2-2\tilde\kappa})
+\mathcal{O}(\eps^{3k-17/2-2\tilde\kappa})
+\mathcal{O}(\eps^{k-3/2})\Big{]}
\Big{[}1+\mathcal{O}(\eps^{k-5/2})\Big{]}dt + d\mathcal{A}_s \\
& = \Big{[}\eps^2c(\xi,\eps) +\mathcal{O}(\eps^{k-1/2}\delta^2) +
\mathcal{O}(\eps^{\min\{2k-11/2-2\tilde\kappa,
3k-17/2-2\tilde\kappa,k-3/2\}})\Big{]} dt + d\mathcal{A}_s.
\end{split}
\end{equation}
Here, recall \eqref{cest}, i.e.,\
$c(\xi,\eps)=\mathcal{O}(\delta^2)$.
\end{proof}

In the following theorem we shall evaluate the noise effect in
the local in time stochastic dynamics \eqref{adyn} driven by the
additive term $d\mathcal{A}_s$ defined by \eqref{def:dAt}, which
supplies the deterministic dynamics of \cite{acfu} with an extra
deterministic drift and a noise term.

\begin{theorem}
\label{neffect}
Assume that in \eqref{a4.1} $k> 5/2$ and $w(0)
\in\mathcal{N}_{L^2}^{\eta\eps^{k-2}}$. Then, locally in time
(as long as $w(t) \in\mathcal{N}_{L^2}^{C\eps^{k-2}}$) the
noise induced terms  appearing in \eqref{adyn}
given by
\begin{equation*}
d\mathcal{A}_s=A^{-1}\Big{[}\frac{1}{2}(v,\partial_\xi^3
(u))-\frac{3}{2}(\partial_\xi^2 u,\partial_\xi u)\Big{]}
(\mathcal{Q}\sigma,\sigma)dt+A^{-1}(\sigma,\mathcal{Q}
\partial_\xi^2 u)dt+A^{-1}(\partial_\xi u,dW),
\end{equation*}
are estimated by
\begin{equation}\label{asest}
d\mathcal{A}_s
=\mathcal{O}(\eta_1)dt
+(\mathcal{O}_{L^2}(\eps^\frac{1}{2}),dW).
\end{equation}
Here, $\eta_1=\|Q\|$ is
the noise strength.

The leading order term in $\mathcal{O}(\eta_1)$ is
$$A^{-2}(\partial_\xi u,\mathcal{Q} \partial_\xi^2 u) +
\mathcal{O}(\eta_1\eps^{\min\{1,k-5/2\}}).$$
\end{theorem}

We can summarize the last two theorems in the next result.

\begin{corollary}
\label{cor:sde}
 Under the assumptions of Theorem \ref{stdyn} 
 it holds that
\begin{equation}
\label{adyn*}
\begin{split}
 d\xi=\eps^2c(\xi(t),\eps)dt
 + & A^{-2}(\partial_\xi u,\mathcal{Q} \partial_\xi^2 u) dt
 +A^{-1}(\partial_\xi u ,dW)
 \\&+\mathcal{O}(\eps^{\min\{2k-11/2-2\tilde\kappa,3k-17/2-2\tilde\kappa,k-1/2\}}
 +\eta_1\eps^{\min\{1,k-5/2\}}) dt.
\end{split}
\end{equation}
\end{corollary}

\begin{proof}
We shall use some estimates proven in Section 4. 
For $\eps$ small, by \eqref{estder}, \eqref{eex},
and \eqref{estx} given in Theorems \ref{fest} and \ref{mainest},
it holds that
\begin{equation*}
\begin{split}
&|(\partial_\xi^2 u,\partial_\xi u)|=\mathcal{O}(\eps^{-1}), \quad
\|\partial_\xi u\|=\mathcal{O}(\eps^{-\frac{1}{2}}),\\
&\|\partial_\xi^2 u\|=\mathcal{O}(\eps^{-\frac{3}{2}}),\quad
\|\partial_\xi^3 u\|= \mathcal{O}(\eps^{-\frac{5}{2}}),\quad
\|u\|_{L^\infty}=\mathcal{O}(1).
\end{split}
\end{equation*}
We need to estimate, by means of upper bounds in terms of $\eps$,
all the new terms appearing in the stochastic dynamics. By Lemma
\ref{leminv} we have
$$|A^{-1}|= \mathcal{O}(\eps),
$$
as locally in time we assumed $\|v\|_{L^2}\leq C\eps^{k-2}.$

Recall from \eqref{help11}
\begin{equation*}
|A^{-1}|\|\partial_\xi u\|^2= 1+\mathcal{O}(\eps^{k-5/2}).
\end{equation*}
We estimate the variance
$\sigma$ by
$$
\|\sigma\|^2
=|A^{-1}|^2\|\partial_\xi u\|^2
=|A^{-1}||A^{-1}|\|\partial_\xi u\|^2
\leq C|A^{-1}|
\leq C\eps\;,
$$
and therefore, deduce that
$$
\|\sigma\|\leq C\eps^{\frac{1}{2}}.
$$

The term $A^{-1}(\partial_\xi^2 u,\partial_\xi
(u))(\mathcal{Q}\sigma,\sigma)$ is
estimated as follows:\\
For the induced $L^2(\Omega_\delta)$-operator norm $\|\cdot\|$,
we get
$$|(\mathcal{Q}\sigma,\sigma)|
\leq \|\mathcal{Q}\sigma\|\|\sigma\|
\leq \|\mathcal{Q}\|\|\sigma\|^2
\leq C \eta_1\eps .
$$
In addition, we have
\begin{equation*}
A^{-1}(\partial_\xi^2 u,\partial_\xi u)(\mathcal{Q}\sigma,\sigma)
\leq C\eps\eps^{-1}\|\mathcal{Q}\|\eps
= C \eps\eta_1.
\end{equation*}
Further, the following estimate holds true
\begin{equation*}
\begin{split}
|A^{-1}(\sigma,\mathcal{Q} \partial_\xi^2 u)|
\leq |A^{-1}||A^{-1}|\|\partial_\xi u\|\|\mathcal{Q}\|\|\partial_\xi^2 u\|
\leq C\eps^{2}\eps^{-1/2}\|\mathcal{Q}\| \eps^{-3/2}
= C \eta_1.
\end{split}
\end{equation*}
The remaining term depending on $v$ is bounded (as long as  $\|v\|_{L^2}\leq C\eps^{k-2}$)
by
\[
\frac{1}{2} |A^{-1}(v,\partial_\xi^3(u))(\mathcal{Q}\sigma,\sigma)|
\leq C \eps \|v\| \eps^{-5/2} \|Q\| \|\sigma\|^2 = C \eps^{k-5/2} \eta_1\;.
\]
This finishes the proof of the theorem.
\end{proof}


\subsection{Interpretation of the result}

Here, we present some comments on the results derived in the
section above.
\begin{remark}[It\^{o}-Stratonovich correction]
\label{rem:strato}
Let us look more closely on Corollary \ref{cor:sde}.
Using the It\^{o}-Stratonovich correction term, we obtain for a function $g(\xi)$
that the Stratonovich differential is
\[
(g,\circ dW) =  \frac12 (g,  \mathcal{Q} \partial_\xi g ) dt + (g, dW)\;,
\]
 where $\mathcal{Q}$ is the covariance operator of $W$.

 Thus, in our case we use that $\partial_\xi v=\partial_\xi u$,
 as the solution $w$ is independent of $\xi$, in order to obtain
 after some calculation
 \[
 d\mathcal{A}_t =  A^{-1}(\partial_\xi u,\circ dW)\;,
\]
which is the Wiener process $W$ projected onto the manifold
 of the small droplets. Here $\circ dW$ denotes the Stratonovic differential.
\end{remark}

\begin{remark}[Boundaries of constant curvature]
\label{rem:curv} If the noise is small, which we will need for
the attractivity result of the manifold, then the motion of the
droplet is to first order given by the $ d\xi = \eps^2c(\xi,\eps)
dt$, which is the deterministic result of \cite{acfu}. On
timescales of order $\eps^{-2}$ the droplet moves with velocity
determined by the changes in the curvature of the boundary.

If our domain has parts of \textit{constant curvature}, like a
\textit{circle} or the \textit{sides on a square}, then
$c(\xi,\eps)$, which depends essentially on the derivative of the
curvature, is $0$ and we expect the droplet to move with a Wiener
process projected to the manifold, i.e.,\ locally like a Brownian
motion. Nevertheless, here we need to look more closely into the
higher order (in $\delta$) corrections to $c(\xi,\eps)$, in order
to prove such a claim.
\end{remark}

\begin{remark}[Extremal points of curvature]
Similarly  to Remark \ref{rem:curv}, we could study the random fluctuations at the
stationary points of maximal or minimal curvature. In this case, at least
locally, we expect also the Brownian motion to dominate.
At points of minimal curvature the noise drives the droplet away 
from this unstable stationary situation.
At points of maximal curvature the droplet is deterministically
attracted, and noise only induces fluctuations around that point.

Moreover, an exit result of large deviation type from this point 
of maximal curvature should hold for the droplet.
But it is an interesting question, in which direction the droplet will exit.
For small amplitude noise, we conjecture that it moves along the
manifold and not exits in normal direction away from
boundary. Our present stability result does not answer this
question as the exponential time-scales present in 
large deviation problems are too long for our result.
\end{remark}

\begin{remark}[Large Noise]
\label{rem:LN} Our approximation of $\xi$ is valid for a general
noise strength $\eta_1:=\|\mathcal{Q}\|$ as long as
$\|v\|_{H_\eps^1}=\mathcal{O}(\eps^{k-(3^+)})$. If the noise
strength is large, we expect the droplet to move randomly,
independent of the curvature of the boundary.

But as we shall see in the sequel, we are not yet able to verify
the stability of the slow manifold for relatively large noise
strength. Only for a sufficiently smooth in space noise of
sufficiently small strength, the attractivity result is
established on arbitrarily long time intervals. As we prove, the
restrictions that we must impose on the noise strength for
maintaining $\|v\|_{H_\eps^1}=\mathcal{O}(\eps^{k-3-})$, lead
to a noise that does not dominate the deterministic dynamics.

More specifically, as we prove, for any $k> 5$ one of the
restrictions is given by
$$\eta_1\leq\eta_0=\mathcal{O}(\eps^{2k - 2-}),$$  resulting in a noise
 which is small
when compared with the original  dynamics of the deterministic problem.

This is mainly due to the fact that we are not able yet to find an
efficient way to control the nonlinearity further away from the
manifold of droplet states. Moreover, the linear attractivity of
the manifold is pretty weak. If one could overcome either one of
these difficulties, then the case of a much larger noise strength
could be analyzed; this is a difficult open problem to be
investigated in a future work.

\end{remark}

%
\section{Stochastic stability}\label{sec3}
%
%
%

In this section, we establish the stability of our
droplet-manifold in the $L^2$-norm for any polynomial times of
order $\mathcal{O}(\eps^{-q})$, $q>0$. Note that we could not
apply standard large deviation type estimates, since we are not
exiting from a single fixed point, and as we only have a manifold
of approximate solutions. Moreover, results in the spirit of
Berglund and Gentz \cite{BG:book} are not yet developed in the
infinite dimensional setting.

%
\subsection{\texorpdfstring{$L^2$}--bounds}
%
Here, we follow the proof of \cite{acfu} (p. 296) with some
significant changes related to the additive noise.

We wish to show that some small tubular neighborhood of our
manifold is positively invariant. \textit{Obviously, in the
stochastic setting any solution will leave the neighborhood at
some point. The question is, how long does this take}. Here, we
are going to present a relatively simple and direct proof for the
bound on the exit time.

Recall that
\begin{equation*}
w(t)=u(\cdot,\xi(t),\eps)+v(t)\;,
\end{equation*}
where
$v(t)\bot_{L^2(\Omega_\delta)} \partial_\xi u$. Relation
\eqref{4**} gives
\begin{equation}\label{4}
\begin{split}
dw
=\Big{[}\mathcal{L}^\eps(u)-Lv-N(u,v)\Big{]}dt+dW,
\end{split}
\end{equation}
and on the other hand
\begin{equation}\label{4b}
 dw=du+dv=\partial_\xi ud\xi+\frac{1}{2}\partial_\xi^2 u(Q\sigma,\sigma)dt + dv\;.
\end{equation}
Recall the definitions
\begin{equation*}
\begin{split}
\mathcal{L}^\eps(u):=&\eps^2\Delta
u-f(u)+\frac{1}{|\Omega_\delta|}\int_{\Omega_\delta}f(u)dx,
\qquad\qquad
Lv:=-\eps^2\Delta v+f^{\prime}(u)v,\\
N(u,v):=&f(u+v)-f(u)-f^{\prime}(u)v-\frac{1}{|\Omega_\delta|}\int_{\Omega_\delta}[f(u+v)-f(u)]dx.
\end{split}
\end{equation*}
Recall also from Lemma \ref{lem:defu} that (for any large $K>k$)
\begin{equation}\label{jul1}
\mathcal{L}^\eps(u)=c(\xi,\eps)\eps^2\partial_\xi u+\mathcal{B},
\quad \text{with }\mathcal{B}=\mathcal{O}_{L^\infty}(\eps^K).
\end{equation}
Solving \eqref{4}  and \eqref{4b} for $dv$ and substituting
\eqref{jul1}, we obtain the equation for motion orthogonal to the manifold.
\begin{lemma}
Consider a solution $w(t)=u(\cdot,\xi(t),\eps)+v(t)$ with
$v(t)\bot \partial_\xi u$ and $\xi$ being our diffusion process, then
\begin{equation}\label{j4}
\begin{split}
dv=&\Big{[}c\eps^2\partial_\xi u
+\mathcal{B}
-Lv-N(u,v)\Big{]}dt
-\partial_\xi ud\xi
-\frac{1}{2}\partial_\xi^2 u(\mathcal{Q}\sigma,\sigma)dt
+dW.
\end{split}
\end{equation}
\end{lemma}

Let us now turn to the estimate of $\|v\|^2$.  First we obtain since $(\partial_\xi u,v)=0$
\begin{equation}\label{j5}
\begin{split}
(dv,v)
&=\Big{[}(\mathcal{B},v)-(Lv,v)-(N(u,v),v)\Big{]}dt-\frac{1}{2}(\partial_\xi^2
(u),v)(\mathcal{Q}\sigma,\sigma)dt+(v,dW).
\end{split}
\end{equation}

In view of \eqref{j5} we first observe that It$\rm{\hat{o}}$
calculus gives
$$
d\|v\|^2=d(v,v)=2(v,dv)+(dv,dv).
$$
Since
$d\xi=bdt+(\sigma,dW)$ and
$$
(dW,dW) = \text{trace}(Q) dt = \eta_0dt,
$$
again by It\^o-calculus we derive
\begin{equation*}\label{477}
\begin{split}
(dv,dv)=& (dW,dW) - 2(dW,\partial_\xi u)d\xi+(\partial_\xi u,\partial_\xi u)d\xi d\xi\\
=&\eta_0dt-2(\partial_\xi
(u),\mathcal{Q}\sigma)dt+\|\partial_\xi u\|^2(\sigma,\mathcal{Q}\sigma)dt \\
=& \mathcal{O}(\eta_0) dt,
\end{split}
\end{equation*}
where we used in the last step that $\|\partial_\xi u\|\leq C\eps^{-1/2}$,
$\|\mathcal{Q}\|=\eta_1\leq \eta_0$ and $\|\sigma\|\leq
C\eps^{1/2}$.

Finally, we have
\begin{equation}\label{j6}
 \frac{1}{2}d\|v\|^2 = (v,dv)+ \mathcal{O}(\eta_0) dt.
\end{equation}
In order to proceed, we now bound the terms in \eqref{j5}.
Obviously, we have
\begin{equation}\label{j7}
(\mathcal{B},v)dt\leq C\eps^K\|v\|dt.
\end{equation}
Furthermore, for the quadratic form of the linearized operator
the following lemma based on \cite{acfu} holds true.
\begin{lemma}
 There exists some $\nu_0>0$ such that for all $v\perp \partial_\xi u$
\begin{equation}\label{j8}
-(Lv,v)\leq -\nu_0\eps^2\|v\|_{H_\eps^1}^2.
\end{equation}
\end{lemma}
Let us remark that any improvement in the spectral gap immediately yields an improvement in the
noise-strength we can study.
\begin{proof}
From the main spectral theorem of \cite{acfu} we have
$$-(Lv,v)\leq -\tilde\nu\eps^2\|v\|^2.$$ As $u$ is bounded, this
implies that for any $\gamma\in(0,1)$
\[-(Lv,v)\leq
-\gamma  \eps^2\|\nabla v \|^2 + \gamma \| f'(u) \|_\infty \|v\|^2 -(1-\gamma)\tilde\nu\eps^2\|v\|^2.
\]
Choosing $\gamma=\eps^2$ yields the claim.
\end{proof}
We consider  finally the term $-(N(u,v),v)dt$.
Since $u$ is uniformly bounded, we obtain
\begin{equation*}
\begin{split}
(N(u,v),v)=\int_{\Omega_\delta}(3uv^3+v^4)dx-\frac{1}{|\Omega_\delta|}\int_{\Omega_\delta}
(3u^2v+3uv^2+v^3)dx\int_{\Omega_\delta} vdx\geq
-C\int_{\Omega_\delta}|v|^3dx+\int_{\Omega_\delta}v^4dx,
\end{split}
\end{equation*}
where we used that the mass conservation
together with the definition of $u$ gives
$$\frac{1}{|\Omega_\delta|}\int_{\Omega_\delta}wdx=\frac{1}{|\Omega_\delta|}\int_{\Omega_\delta}udx,
\qquad\text{and therefore}\qquad
\int_{\Omega_\delta}v dx=0.
$$
Hence, we arrive at
\begin{equation}\label{j9}
-(N(u,v),v)dt\leq
\Big{(}C\int_{\Omega_\delta}|v|^3dx-\int_{\Omega_\delta}v^4dx\Big{)}dt.
\end{equation}
Also we have
\begin{equation}\label{j10}
 (\partial_\xi^2 u,v)(\mathcal{Q}\sigma,\sigma)dt\leq
 C\eps^{-3/2}\eta_1\eps^{1/2}\eps^{1/2}\|v\|dt=C\eps^{-1/2}\eta_1\|v\| dt.
 \end{equation}
Using now the relations \eqref{j6}, \eqref{j7}, \eqref{j8},
\eqref{j9} and \eqref{j10} in \eqref{j5} we get
\begin{equation}\label{j10*}
\begin{split}
d\|v\|^2\leq
&\Big{[}C\eta_0+C\eps^K\|v\|-2\nu_0\eps^2\|v\|_{H_\eps^1}^2
+C\int_{\Omega_\delta}|v|^3dx-2\int_{\Omega_\delta}v^4dx+C\eps^{-1/2}\eta_1\|v\|\Big{]}dt +2(v,dW).
\end{split}
\end{equation}
By Nirenberg's inequality and since $\eps \ll 1$, it follows that
\begin{equation}\label{j10**}
\begin{split}
C\int_{\Omega_\delta}|v|^3dx=&C\|v\|_{L^3}^3
\leq
\tilde{C}\|v\|_{H_1}\|v\|^{2}
\\
\leq &C_2\frac{1}{\eps}\|v\|_{H_\eps^1}\|v\|^2
\leq
\|v\|_{H_\eps^1}\|v\|\nu_0\eps^2
\leq \nu_0\eps^2\|v\|_{H_\eps^1}^2,
\end{split}
\end{equation}
provided that 
for $\hat{C}\geq C_2$
\begin{equation}\label{loc*}
\|v\|\leq \frac{\nu_0 \eps^3}{\hat{C}}.
\end{equation}
Note that this is the point where the condition $k>5$ will
finally appear, since if $\|v\|\leq C\eps^{k-2}$ for $k>5$ then
indeed $\|v\|\leq \frac{\nu_0 \eps^3}{\hat{C}}$ (observe that we
do not have any sharp estimate of the constant $\nu_0$ even if
$\hat{C}$ could be determined explicitly by the Nirenberg's inequality
constant).

Using \eqref{j10**} in \eqref{j10*} and the fact that
$-\|v\|_{H_\eps^1}^2\leq -\|v\|^2$ we arrive at
\begin{equation*}
\begin{split}
d\|v\|^2 \leq
&\Big{[}C\eta_0+C\eps^K\|v\|-\nu_0\eps^2\|v\|^2_{H^1_\eps}+C\eps^{-1/2}\eta_1\|v\|\Big{]}dt+2(v,dW)\\
\leq
&\Big{[}C\eta_0+\frac{C_0^2}{2\nu_0\eps^2}(\eps^K+\eps^{-1/2}\eta_1)^2-\frac{\nu_0\eps^2}{2}\|v\|^2\Big{]}dt+2(v,dW).
\end{split}
\end{equation*}
Let us summarize the result proven so far.
\begin{lemma}
As long as $ \|v\|\leq \frac{\nu_0 \eps^3}{\hat{C}},$ then
\begin{equation}\label{j11}
\begin{split}
d\|v\|^2\leq
&\Big{[}C_\eps-\frac{\nu_0\eps^2}{2}\|v\|^2\Big{]}dt+2(dW,v),
\end{split}
\end{equation}
for
$$C_\eps:=C\eta_0+\frac{C_0^2}{2\nu_0\eps^2}(\eps^K+\eps^{-1/2}\eta_1)^2,$$
where $C_0$, $\hat{C}$ and $C$ are the specific constants appearing
in the proof.
\end{lemma}


\subsection{Long-time \texorpdfstring{$L^2$}--stability}

Let us define the stopping time $\tau^\star$ as the exit-time of a
neighborhood of the manifold before time $T$
$$
\tau^\star:= \inf\{t\in[0,T]\ :\ \|v(t)\|>B \},
$$
with the convention that $$\tau^\star=T,\quad\text{ if } \quad
\|v(t)\|\leq B \text{ for all }t\in[0,T],$$ and thus, the solution
did not exit before $T$.

Recall that $v$ satisfies
an inequality of the form
$$
d\|v\|^2 \leq [C_\eps-a\|v\|^2] dt + 2 (v,dW),
$$
 for all $t\leq \tau^\star$, provided that $B \leq \frac{\nu_0 \eps^3}{\hat{C}}$.
 Thus from now on, we fix
 \[B=C\eps^{k-2} \quad\text{for }k>5\;.
 \]
 More specifically,  from \eqref{j11}
\begin{equation*}
a:=\frac{\nu_0\eps^2}{2}
\quad\text{and}\quad
C_\eps:=C\eta_0+\frac{C_0^2}{2\nu_0\eps^2}(\eps^K+\eps^{-1/2}\eta_1)^2
\end{equation*}
So, for all $t\leq \tau^\star$, it holds that
$$
\|v(t)\|^2 + a \int_0^t\|v(s)\|^2 ds  \leq  \|v(0)\|^2 + C_\eps t + 2
\int_0^t (v,dW ).
$$
Using the fact that stopped stochastic integrals still have mean
value $0$ (referring to optimal stopping of martingales), we
obtain
\begin{equation}
\label{e:key1}
 \mathbb{E}\|v(\tau^\star)\|^2
+ a \mathbb{E}\int_0^{\tau^\star}\|v(s)\|^2 ds  \leq  \|v(0)\|^2
+  C_\eps T,
\end{equation}
where we used that $\tau^\star\leq T$ by definition.

We can extend (\ref{e:key1}) above  to higher powers using It\^o
calculus as follows:
\begin{eqnarray*}
 d\|v\|^{2p}
 &=& p \|v\|^{2p-2} d\|v\|^2 + p(p-1) \|v\|^{2p-4} d\|v\|^2 d\|v\|^2 \\
 &=& p \|v\|^{2p-2} d\|v\|^2 + 4p(p-1) \|v\|^{2p-4} ( v, Qv ) dt\\
  &\leq& p \|v\|^{2p-2} [C_\eps-a\|v\|^2] dt  + 4p(p-1) \|v\|^{2p-2} \|Q\| dt + 2p \|v\|^{2p-2} (v,dW
  ).
\end{eqnarray*}
Hence, for all integers $p>1$, we arrive at
\[
\|v(t)\|^{2p} + p a \int_0^t\|v(s)\|^{2p} ds
\leq  \|v(0)\|^{2p} + C (C_\eps+ \|Q\|)\int_0^t \|v\|^{2p-2} dt +
2p  \int_0^t \|v\|^{2p-2} ( v,dW ),
\]
 provided $t\leq \tau^\star$.
Here and in the sequel, we denote all constants depending
explicitly on $p$ only by $C$.

Therefore, we obtain
\begin{equation}
 \label{e:key2}
 \mathbb{E}\|v(\tau^\star)\|^{2p}
\leq  \|v(0)\|^{2p} + C
(C_\eps+\|Q\|)\mathbb{E}\int_0^{\tau^\star}\|v\|^{2p-2} dt,
\end{equation}
and
\begin{equation}
 \label{e:key3}
 a \mathbb{E}\int_0^{\tau^\star} \|v(s)\|^{2p} ds
\leq  \frac1p\|v(0)\|^{2p} + C
(C_\eps+\|Q\|)\mathbb{E}\int_0^{\tau^\star}\|v\|^{2p-2} dt.
\end{equation}

Let us now assume that
$$
q= \frac{C_\eps+\|Q\| }{a} \ll 1 \qquad \text{and} \qquad \|v(0)\|^2
\leq  q \ll B^2.
$$

An induction argument yields
\begin{align*}
\tfrac1p\mathbb{E} \|v(\tau^\star)\|^{2p}
 &\stackrel{(\ref{e:key2})}{\leq}
 \tfrac1p  \|v(0)\|^{2p}
+ C (C_\eps+\|Q\|)\mathbb{E}\int_0^{\tau^\star}\|v\|^{2p-2} dt \\
 &= \tfrac1p  \|v(0)\|^{2p}
+    C \cdot  qa \cdot\mathbb{E}\int_0^{\tau^\star}\|v\|^{2p-2} dt\\
 &\stackrel{(\ref{e:key3})}{\leq}
 \tfrac1p   \|v(0)\|^{2p} +  C q \|v(0)\|^{2p-2} +   C q^2 a \mathbb{E}\int_0^{\tau^\star}\|v\|^{2p-4} dt\\
&\leq   C q \|v(0)\|^{2p-2} +  C q^2 a \mathbb{E}\int_0^{\tau^\star}\|v\|^{2p-4} dt\\
&\leq \ldots \\
& \leq   C q^{p-2} \|v(0)\|^{4} +  C q^{p-1} a \mathbb{E}\int_0^{\tau^\star}\|v\|^{2} dt \\
 &\stackrel{(\ref{e:key1})}{\leq}  C q^{p-2} \|v(0)\|^{4} +  C q^{p-1} [ \|v(0)\|^{2} +C_\eps T]\\
& \leq   C q^p   + C a q^p T,
\end{align*}
as $C_\eps \leq aq$.
Using Chebychev's inequality, finally, we arrive at
\begin{equation}\label{newche}
\begin{split}
 \mathbb{P} (\tau^\star < T )
 & =   \mathbb{P} (\|v(\tau^\star)\| \geq B )
  \leq B^{-2p} \mathbb{E} \|v(\tau^\star)\|^{2p} \\
  & \leq C B^{-2p}[ q^p   + a  q^p T]
  = C \Big(\frac{q}{B^2}\Big)^p   + C a   \Big(\frac{q}{B^2}\Big)^p
  T.
  \end{split}
\end{equation}
Therefore, we obtain the following $L^2$-stability theorem.

\begin{theorem}\label{newl2thm}
Consider the exit time
$$
\tau^\star:= \inf\{t\in[0,T_\eps]\ :\ \|v(t)\|>C\eps^{k-2}  \},
$$
with $T_\eps:=\varepsilon^{-N}$ for any fixed large $N>0$. Fix
$$k> 5,\;\;\;\;\mbox{and}\;\;\;\;\|v(0)\|< \eta\eps^{k-2}.
$$
Also, assume that the noise satisfies
$$\eta_0\leq C\eps^{2k-2+\tilde{k}},
$$ for
some $\tilde{k}>0$ very small. Then the probability $\mathbb{P}
(\tau^\star < T_\eps )$ is smaller than any power of
$\varepsilon$, as $\varepsilon \to 0$. And thus for very large
times with high probability the solution stays close to the
manifold.
\end{theorem}
\begin{proof}
The claim follows from inequality \eqref{newche} if $q/B^2 = \mathcal{O}(\varepsilon^{\tilde{\kappa}})$.

Indeed, using  the
definitions of $C_\eps$, $a=\mathcal{O}(\eps^2)$ and $B=C\eps^{k-2}$, we have
(note $K>k$)
$$q= \frac{C_\eps+\|Q\| }{a}\leq
C [\eta_0\eps^{-2}+\eps^{-4}(\eps^{2K}+\eps^{-1}\eta_1^2)]$$
 Furthermore,
$B^2=\mathcal{O}(\eps^{2k-4})$. So, indeed we get
$$q/B^2 =  \eta_0\eps^{2-2k}+\eps^{2(K-k)}+\eps^{-1-2k}\eta_1^2   =  \mathcal{O}(
\varepsilon^{\tilde{\kappa}}),
$$
since $\eta_1\leq\eta_0$ and $k>5$.
\end{proof}

\begin{remark}
The stability result presented so far does not state that
the local in time stochastic dynamics for $\xi$ given by Theorem
\ref{stdyn} and Theorem \ref{neffect} hold with high probability
for a long time.
For this we need to prove stochastic stability in the
$H_\eps^1$-norm. As we rely for simplicity of presentation on the
direct application of It\^o's formula, this will only
be achieved for a noise $\dot{W}$ sufficiently regular
in space.
\end{remark}

\begin{remark}
The presence of the very small $\tilde\kappa$ in the conditions of
$\eta_0$ is only for simplicity. Instead of having the
$\tilde\kappa$ here, we could use $B:=\eps^{2k-4-\tilde\kappa}$.
This would yield the same result but for a slightly smaller
neighborhood for $v$.

We need the gap created by $\tilde\kappa$ in order to control the
probability and to obtain very large time-scales in the stability
result. This is the reason, why $k=5$ (included in the
deterministic case result of \cite{acfu}) is out of reach in our
approach, and we can only consider $k>5$.
\end{remark}


\subsection{Estimates in \texorpdfstring{$H_\varepsilon^1$}--norm}
\label{sec4new}

As we can not rely on bounds of the linearized operator in
$H^1_\eps$-norm, we shall use instead the previously established
$L^2$-stability result given by Theorem \ref{newl2thm}. Nevertheless,
in order to bound the $H^1_\eps$-norm of stochastic
solution over a very long time-scale, we can allow the use of a
larger tube bounding $\nabla v$. But, as we shall see in the sequel,
this will further limit the noise  we can
consider since we obtain additional restrictions for the size of
$\eta_2$.

Let $w$ be the solution of the mass conserving stochastic
Allen-Cahn equation \eqref{equ}, then
$$
w=u+v,\;\;u\in\mathcal{M},\;\;u\bot_{L^2(\Omega_\delta)}\partial_\xi u.
$$
Moreover, recall (\ref{j4})
\begin{equation*}
\begin{split}
dv=&\Big{[}c\eps^2\partial_\xi u+\mathcal{B}-Lv-N(u,v)\Big{]}dt
-\partial_\xi ud\xi-\frac{1}{2}\partial_\xi^2 u(\mathcal{Q}\sigma,\sigma)dt+dW,
\end{split}
\end{equation*}
with
\[
Lv+N(u,v) =-\eps^2\Delta v + f(u+v)-f(u)-\frac{1}{|\Omega_\delta|}\int_{\Omega_\delta}[f(u+v)-f(u)]dx.
\]
We consider first the following relation
\begin{equation}
\label{e:itostart}
\begin{split}
 d\|\nabla v\|^2
 &=  2(\nabla v, d \nabla v) + (\nabla dv, \nabla dv)\\
  &= - 2(\Delta v, dv) + (\nabla dv, \nabla dv),
\end{split}
\end{equation}
where we used integration by parts, as $v$ satisfies a Neumann
boundary condition.

Observe that by series expansion of $W$ and since $\|e_k\|=1$,
$\|\sigma\|=\mathcal{O}(\eps^{1/2})$, $\|\nabla\partial_\xi
u\|=\mathcal{O}(\eps^{-3/2})$, we obtain
\begin{equation*}
\begin{split}
(\nabla\partial_\xi u, \nabla dW)(\sigma,dW) &= \sum_k \alpha_k^2
(\nabla\partial_\xi u, \nabla e_k)(\sigma,e_k) dt\\
&\leq 
\sum_k \alpha_k \|\nabla e_k\| \alpha_k
\|\nabla\partial_\xi u\| \|\sigma\|dt \\
&\leq  \eta_2^{1/2} \eta_0^{1/2} \|\nabla\partial_\xi u\|
\|\sigma\|dt\\
&\leq C(\eps^{-2}\eta_0 + \eta_2) dt.
\end{split}
\end{equation*}
Thus, since $\eta_1\leq \eta_0$ then for the It\^o-correction term
we have
\begin{equation*}
 \begin{split}
 (\nabla dv, \nabla dv)
 &= \|\nabla\partial_\xi u\|^2 (d\xi)^2
 -2 (\nabla \partial_\xi u,  \nabla dW)  d\xi
 + (\nabla dW, \nabla dW) \\
  &= \|\nabla\partial_\xi u\|^2 (Q\sigma,\sigma) dt  -2 (\nabla \partial_\xi u,  \nabla dW)(\sigma,dW) + \text{trace}(\Delta Q)\\
  &= \mathcal{O}(\eps^{-2}\eta_1 + \eta_2+\eps^{-2}\eta_0 + \eta_2) dt\\
  &=\mathcal{O}(\eps^{-2}\eta_0 + \eta_2) dt.
\end{split}
\end{equation*}
Considering the other mixed term in \eqref{e:itostart} and using
that $(1,\Delta v)=0$, we obtain the  following relation
\begin{equation}\label{nnn1}
 \begin{split}
-(\Delta v, dv) = & -\Big( \Delta v, \mathcal{B}+\eps^2\Delta v \Big) dt
- \Big(  \Delta v, - f(u+v)+f(u)\Big)    dt \\
&+ \Big(\Delta v, \partial_\xi u [b- c\eps^2] +\frac{1}{2}\partial_\xi^2 u(\mathcal{Q}\sigma,\sigma)\Big) dt
- \Big(\Delta v, -\partial_\xi u (\sigma,dW) + dW\Big) \\
=&T_1+T_2+T_3+T_4\;.
\end{split}
\end{equation}

First, we estimate the martingale term $T_4$. Note that we only integrate by parts once,
as we only know that $v$ satisfies Neumann boundary conditions.
\[
 \begin{split}
T_4 &= (\Delta v, \partial_\xi u) (\sigma,dW)
- (\Delta v,dW) 
= -( \nabla v, \nabla\partial_\xi u) (\sigma,dW)
+ (\nabla v,\nabla dW)\\
& = (\mathcal{O} (\eps^{-1}\| \nabla v\|), dW) +
(\mathcal{O}(\|\nabla v\|), d\nabla W).
\end{split}
\]
For $T_1$ and for $K>0$ sufficiently large, we obtain by Lemma \ref{lem:defu} that
\[
 \begin{split}
T_1 &=
- ( \Delta v, [\mathcal{B}+\eps^2\Delta v] ) dt 
=
-( \Delta v, \mathcal{B}) dt
-\eps^2 \| \Delta v\|^2 dt \\
& = \mathcal{O}(\eps^{K}\|\Delta v\| ) dt -\eps^2 \| \Delta v\|^2
dt.
\end{split}
\]
For $T_3$ we have
\[ \begin{split}
   T_3&=
\Big(\Delta v, \partial_\xi u [b-c\eps^2] +\frac{1}{2}\partial_\xi^2 u(\mathcal{Q}\sigma,\sigma)\Big) dt
=
(\Delta v, \partial_\xi u)  [b-c\eps^2] dt +\frac{1}{2} ( \partial_\xi^2 u, \Delta v )(\mathcal{Q}\sigma,\sigma) dt\\
& = \mathcal{O} (\|\Delta v\| \eps^{-1/2} [b-c\eps^2]  )  dt + \mathcal{O} ( \eps^{-1/2} \eta_0 \|\Delta v\|)  dt.
   \end{split}
\]
Here, note that we can bound $b$ by Corollary \ref{cor:sde}
only up to a stopping time.

For $T_2$ we derive
\begin{equation*}
 \begin{split}
 T_2&=-\Big( \Delta v,  - f(u+v)+f(u) \Big)dt
 = \Big( \Delta v, -v + 3u^2v+3uv^2+v^3 \Big)dt \\
  &= \Big{(}\|\nabla v\|^2 - \int  \nabla(3u^2v+3uv^2+v^3) \cdot \nabla v dx\Big{)}dt  \\
   &= \Big{(}\|\nabla v\|^2 - \int (3u^2+6u v + 3v^2) |\nabla v|^2 dx - \int (6uv+3v^2)
     (\nabla u \cdot \nabla v) dx\Big{)}dt  \\
&\leq  \Big{(}\|\nabla v\|^2- \int (6uv+3v^2)  (\nabla u \cdot
\nabla v) dx\Big{)}dt.
\end{split}
\end{equation*}
We observe that $\|\nabla u\|_{\infty}\sim \|\partial_\xi
u\|_{\infty}$ (cf.\ Appendix).
So, we have $\|\nabla u\|_{\infty}\leq C\eps^{-1}$ (since
$\|\partial_\xi u\|_{\infty}\leq C\eps^{-1}$, cf.\ \cite{acfu}).
In addition, $u$ is uniformly bounded in $\eps$. Furthermore, the
Nirenberg's inequality gives
$$
\|v^2\|=\|v\|_4^2\leq
(C\|v\|_{H_1}^{1/2}\|v\|^{1/2})^2=C\|v\|_{H_1}\|v\| \leq
C\|v\|^2+C\|\nabla v\|\|v\|,
$$
while the following interpolation inequality holds true
$$
\|\nabla v\|^2=-(\Delta v,v)\leq \|\Delta
v\|\|v\| .
$$
Using the previous estimates we arrive at
\begin{equation}\label{esjull}
\begin{split}
T_2 &\leq\Big{(}\|\nabla v\|^2 +
3\|v^2\|\|\nabla u\|_{\infty}\|\nabla v\|+ 6\|v\| \|u\|_{\infty} \|\nabla u\|_\infty\|\nabla v\|\Big{)}dt\\
&\leq \Big{(}\|\nabla v\|^2+ C\eps^{-1}\|v^2\|\|\nabla v\|
+ C\eps^{-1}\|v\|\|\nabla v\|\Big{)}dt\\
&\leq \Big{(}\|\nabla v\|^2+ C\eps^{-1}\|v\|\|\nabla v\|^2+
C\eps^{-1}\|v\|^2 \|\nabla v\| + C\eps^{-1}\|v\|\|\nabla
v\|\Big{)}dt.
\end{split}
\end{equation}
So, by relation \eqref{nnn1}, we derive
\begin{equation*}
  \begin{split}
   d\|\nabla v\|^2
  =&
\mathcal{O}(\eps^{K}\|\Delta v\| ) dt
-2\eps^2 \| \Delta v\|^2 dt
+ T_2 \\
&+\mathcal{O} (\|\Delta v\| \eps^{-1/2} |b-c\eps^2|  )  dt + \mathcal{O} ( \eps^{-1/2} \eta_0 \|\Delta v\|)  dt
  + \mathcal{O}(\eps^{-2}\eta_0 + \eta_2) dt  \\
&  + (\mathcal{O}(\eps^{-1}\|\nabla v\|), d
W)+\mathcal{O}(\|\nabla v\|,\nabla dW).
 \end{split}
\end{equation*}
for $T_2$ estimated by \eqref{esjull}. Therefore, Young's
inequality yields
\begin{equation}
\label{e:H1:1}
  \begin{split}
   d\|\nabla v\|^2
  =&
-\eps^2 \| \Delta v\|^2 dt
+ T_2 \\
&+\mathcal{O} (\eps^{2K-2}+ \eps^{-3} |b-c\eps^2|^2 + \eps^{-3} \eta_0^2+\eps^{-2}\eta_0 + \eta_2) dt  \\
&  + (\mathcal{O}(\eps^{-1}\|\nabla v\|), d
W)+\mathcal{O}(\|\nabla v\|,\nabla dW).
 \end{split}
\end{equation}
In order to proceed with the estimate of $T_2$ given by
(\ref{esjull}) where the $L^2$-norm of $v$ is also involved, we
shall rely on the $L^2$-stability result proven so far, observing
evolution in time as long as $\|v\|$ is not too large.
\begin{definition}
 Let $k>5$ and $\tilde\kappa>0$ small.  For some given large $T_\eps$, we define
the stopping time
\begin{equation}\label{st5}
\tau_\eps=\inf\{t\in[0,T_\eps] \ : \
\|\nabla v(t)\| > C_0\eps^{k-4-\tilde\kappa}
\text{ or }
\|v(t)\|>C_0\eps^{k-2}\}
\;.
\end{equation}
Here, $C_0$ is a large fixed positive constant. Obviously, we set
$\tau_\eps=T_\eps$ if none of the conditions is satisfied for all
$t< T_\eps$.
\end{definition}

From the previous definition, it follows that
\[
\sup_{t\in[0,T_\eps]}\|v(t)\|_{H^1_\eps} \leq
C_0\eps^{k-3-\tilde\kappa}.
\]
Hence, for this stopping time considered, the bound in $L^2$-norm
provided by Theorem \ref{newl2thm} (i.e.,\ $\|v(t)\|< C\eps^{k-2}$
for all $t\leq \tau_\eps $ with high probability, as long as
$T_\eps$ is a polynomial in $\eps^{-1}$) is much stronger than
the $L^2$-bound given in the $H^1_\eps$-norm.

For the rest of the arguments we assume that $t\leq\tau_\eps$ and let $K \geq k$.
We know by Corollary \ref{cor:sde} that
\[
 \sup_{[0,\tau_\eps]} |b-c\eps^2| = \mathcal{O}(\eps^{k-1/2}).
\]
Thus, relation (\ref{e:H1:1}) yields
\begin{equation}\label{eqn12}
  \begin{split}
   d\|\nabla v\|^2
  =&
-\eps^2 \| \Delta v\|^2 dt
+ T_2 \\
& + \mathcal{O} ( \eps^{2k-4} + \eps^{-3} \eta_0^2+\eps^{-2}\eta_0 + \eta_2) dt  \\
& + (\mathcal{O}(\eps^{-1}\|\nabla v\|), d
W)+(\mathcal{O}(\|\nabla v\|),\nabla dW )\;.
 \end{split}
\end{equation}
In order to estimate the term $T_2$, we use the bound given by
(\ref{esjull}) and the following relations
 $$
 \|v(t)\|\leq C\eps^{k-2},
 \qquad
 \|\nabla v\|^2\leq \|\Delta v\|\|v\|.
 $$
 Therefore, for $k\geq 3$,
we obtain by Young's inequality
\begin{equation}\label{esjul2}
\begin{split}
T_2&\leq \Big{(}\|\nabla v\|^2+ C\eps^{-1}\|v\|\|\nabla v\|^2+
C\eps^{-1}\|v\|^2 \|\nabla v\| + C\eps^{-1}\|v\|\|\nabla
v\|\Big{)}dt\\
&\leq \Big{(}4\|\nabla v\|^2+C\eps^{2k-6}\Big{)}dt\\
&\leq \Big{(}\frac12\eps^2\|\Delta v\|^2+8\eps^{-2}\|v\|^2+C\eps^{2k-6}\Big{)}dt\\
&\leq \Big{(}\frac12\eps^2\|\Delta v\|^2+C\eps^{2k-6}\Big{)}dt\;.
\end{split}
\end{equation}
To close the argument, we need a Poincar\'e type estimate. More
specifically, for any function satisfying Neumann boundary
conditions there exists some positive constant $c>0$ such that
\begin{equation}
\label{H1Poin}
2c \|\nabla v\|^2 \leq \| \Delta v \|^2\;.
\end{equation}
To prove the statement above, first let us denote by $\bar{v}$ the spatial
average of $v$, and then use interpolation and the standard
Poincar\'e inequality. So, we have
\[
\|\nabla v\|^2  = \|\nabla (v-\bar{v})\|^2
\leq  \| v-\bar{v}\| \| \Delta(v-\bar{v})\|
\leq C  \| \nabla(v-\bar{v})\| \| \Delta(v-\bar{v})\|
= C  \| \nabla v\| \| \Delta v\|\;.
\]
Therefore, by \eqref{eqn12}, we obtain the following lemma.
\begin{lemma}\label{lem1h1}
If $k\geq 3$ and $t\leq \tau_\eps$, with $\tau_\eps$ given by
\eqref{st5}, then for $c>0$ the constant appearing in Poincar\'e
inequality \eqref{H1Poin}, the following relation holds true
\begin{equation}\label{eqn13}
d\|\nabla v\|^2+c\eps^2\|\nabla v\|^2dt=\Gamma_\eps
dt+(Z,dW)+(\Psi,d\nabla W),
\end{equation}
for
\begin{equation}\label{eqn14}
\Gamma_\eps:=\mathcal{O}(\eps^{2k-6}+\eps^{-3}\eta_0^2+\eps^{-2}\eta_0+\eta_2),
\end{equation}
and
\begin{equation}\label{eqn15}
\|Z\|^2:=\mathcal{O}(\eps^{-2}\|\nabla
v\|^2),\;\;\;\;\|\Psi\|^2:=\mathcal{O}(\|\nabla v\|^2).
\end{equation}

Furthermore, it holds that
\begin{equation}
\|\nabla v(t)\|^2+c\eps^2\int_0^t\|\nabla v(s)\|^2ds\leq \|\nabla
v(0)\|^2+ \Gamma_\eps T_\eps+\int_0^t[(Z,dW)+(\Psi,d\nabla W)]ds,
\end{equation}
and thus
\begin{equation}\label{eqn16}
\mathbb{E}\|\nabla v(t)\|^2+c\eps^2\mathbb{E}\int_0^t\|\nabla
v(s)\|^2ds\leq \|\nabla v(0)\|^2+ \Gamma_\eps T_\eps.
\end{equation}
Note that by a slight abuse of notation, we identify $\Gamma_\eps$ 
defined in (\ref{eqn14}) with the corresponding $\mathcal{O}$-term.
\end{lemma}
We remark, that as in the $L^2$-case this lemma is only  the first step for proving stability;
higher moments will be derived by an induction argument in the
following section.


\subsection{Long-Time \texorpdfstring{$H^1$}--stability}

Keeping the same assumptions as in Lemma \ref{lem1h1}, we proceed similar to the $L^2$-stability result
by estimating for any integer $p>1$ the $p^\text{th}$-moment of $\|\nabla
v\|^2$.
It\^o calculus yields
\begin{equation}\label{eqn17}
d\|\nabla v\|^{2p}=p\|\nabla v\|^{2p-2}d\|\nabla
v\|^2+p(p-1)\|\nabla v\|^{2p-4}[d\|\nabla v\|^2]^2\;.
\end{equation}
Using now \eqref{eqn13}, we obtain
\begin{equation}\label{eqn18}
[d\|\nabla v\|^2]^2
\leq (Z,\mathcal{Q}Z)dt+(\Psi,\Delta\mathcal{Q}\Psi)dt
+2(Z,dW)(\Psi,d\nabla W).
\end{equation}
Observing now that $d\beta_i d\beta_j=\delta_{ij}$, we get by series expansion and Cauchy-Schwarz
\begin{equation}\label{eqn19}
\begin{split}
(Z,dW)(\psi,d\nabla W)
&=\sum a_i^2(Z,e_i)(\psi,\nabla e_i)
\leq \sum a_i^2\|e_i\|\|\nabla e_i\| \|Z\| \|\psi\|\\
& \leq\|Z\| \|\psi\| \sqrt{ \eta_0\eta_2}
\leq \|Z\|^2  \eta_0 + \|\psi\|^2\eta_2\;.
\end{split}
\end{equation}
By \eqref{eqn18} and \eqref{eqn19}, we arrive at
\begin{equation}\label{eqn20}
\begin{split}
[d\|\nabla v\|^2]^2
&\leq
(\|\mathcal{Q}\|\|Z\|^2+\|\Psi\|^2\|\Delta\mathcal{Q}\|)dt+2\|Z\|\|\Psi\|\sqrt{\eta_0\eta_2}dt\\
&\leq C(\|Z\|^2\eta_0
+\|\Psi\|^2\eta_2)dt.
\end{split}
\end{equation}
Replacing \eqref{eqn20} in \eqref{eqn17} and using \eqref{eqn13},
we get
\begin{equation*}
\begin{split}
d\|\nabla v\|^{2p}\leq& p\|\nabla v\|^{2p-2}d\|\nabla v\|^2+Cp(p-1)\|\nabla v\|^{2p-4}(\|Z\|^2\eta_0
+\|\Psi\|^2\eta_2)dt\\
\leq& p\|\nabla v\|^{2p-2}(\Gamma_\eps-c\eps^2\|\nabla v\|^2)dt\\
&+p\|\nabla v\|^{2p-2}[(Z,dW)+(\Psi,d\nabla W)]\\
&+Cp(p-1)\|\nabla v\|^{2p-4}(\|Z\|^2\eta_0
+\|\Psi\|^2\eta_2)dt.
\end{split}
\end{equation*}
In the  relation above, we now use \eqref{eqn15} and the fact that
$\|v\|\leq C\eps^{k-2}$. Hence, we obtain for any $t\leq
\tau_\eps$
\begin{equation}\label{eqn21}
\begin{split}
d\|\nabla v\|^{2p}
\leq& p\|\nabla v\|^{2p-2}(\Gamma_\eps-c\eps^2\|\nabla v\|^2)dt
+p\|\nabla v\|^{2p-2}[(Z,dW)+(\Psi,d\nabla W)]\\
&+C\|\nabla v\|^{2p-2}\eps^{-2}\eta_0dt +C\|\nabla
v\|^{2p-2}\eta_2dt.
\end{split}
\end{equation}
Here, all appearing constants may depend on $p$.
Integrating \eqref{eqn21} we derive the following
lemma:
\begin{lemma}\label{lem1h2}
Under the assumptions of Lemma \ref{lem1h1}, then for any integer
$p\geq 2$, the following estimate holds true
\begin{equation}\label{eqn22}
\begin{split}
\mathbb{E}\|\nabla
v(\tau_\eps)\|^{2p}+c\eps^2p\mathbb{E}\int_0^{\tau_\eps}\|\nabla
v(t)\|^{2p}dt\leq &\|\nabla
v(0)\|^{2p}+C[\Gamma_\eps+\eps^{-2}\eta_0+\eta_2] \cdot \mathbb{E}
\int_0^{\tau_\eps}\|\nabla v(t)\|^{2p-2}dt.
\end{split}
\end{equation}
Here, $c,C>0$ are constants that may depend on $p\in\mathbb{N}$
and $\Gamma_\eps$ was defined in (\ref{eqn14}).
\end{lemma}

Keeping the same assumptions as these of Lemma \ref{lem1h1}, we
consider a sufficiently small noise such that
\begin{equation}
 \eps^{-2}\eta_0+\eta_2 =\mathcal{O}( \eps^{2k-6}),
\end{equation}
i.e.,\ the main order term of $\Gamma_\eps+\eps^{-2}\eta_0+\eta_2$
is $\eps^{2k-6}$.

Furthermore, in \eqref{eqn22} we define
\[
K_p:=\mathbb{E}\int_0^{\tau_\eps}\|\nabla v(t)\|^{2p}dt,
\]
and thus, we obtain for $p\geq 2$
\begin{equation*}
\begin{split}
K_p \leq& \eps^{-2}\|\nabla
v(0)\|^{2p}+C\eps^{-2}[\Gamma_\eps+\eps^{-2}\eta_0+\eta_2]K_{p-1}
\\
\leq&\eps^{-2}\|\nabla
v(0)\|^{2p}+C\eps^{2k-8}K_{p-1}\\
=:& \eps^{-2}A^p+aK_{p-1},
\end{split}
\end{equation*}
for $$A:=\|\nabla v(0)\|^{2},\;\;\;\;a:=C\eps^{2k-8}<1.$$ So, we
have inductively
\begin{equation*}
\begin{split}
K_p  \leq   \eps^{-2} A^p+aK_{p-1}
&\leq \eps^{-2} A^p+ \eps^{-2} a  K_{p-2} A^{p-1}+ 2 a^2 K_{p-2}\\
& \leq \ldots 
\leq C\eps^{-2} \sum_{i=2}^p  a^{p-i}A^i+Ca^{p-1}K_1.
\end{split}
\end{equation*}
 By definition and  \eqref{eqn16} we get
\[
K_1\leq \eps^{-2}A + \eps^{-2}\Gamma_\eps T_\eps \leq \eps^{-2}A +
Ca T_\eps .
\]
Hence, we obtain
\begin{equation}
K_p \leq C\eps^{-2}\sum_{i=1}^p  a^{p-i}A^i + Ca^{p} T_\eps \leq
C[\eps^{-2}+T_\eps] a^{p} + C \eps^{-2} A^{p},
\end{equation}
for $C>0$ a constant depending on $p$.

This yields the following lemma. 
Note that here we need only $k>3$. It is the $L^2$-bound that needs $k>5$.

\begin{lemma}\label{lem1h3}
Let $k\geq 3$ and $\tau_\eps$ given by \eqref{st5}. If
$$
 \Gamma_\eps +\eps^{-2}\eta_0+\eta_2\leq C\eps^{2k-6},
$$
and if $$\|\nabla v(0)\|^2 \leq a:=C\eps^{2k-8},$$ then for any
integer $p>1$ it holds that
\begin{equation}
\begin{split}
\mathbb{E}\|\nabla v(\tau_\eps)\|^{2p}\leq C \eps^{-2}  [\eps^{-2}+T_\eps] a^{p}.
\end{split}
\end{equation}
\end{lemma}

\begin{proof}
Using the definitions of $a$, $A$, and $K_p$ and relation
\eqref{eqn22}, we have
\begin{equation*}
\begin{split}
\mathbb{E}\|\nabla v(\tau_\eps)\|^{2p}\leq &A^{p} + C\eps^{-2} a K_{p-1}\\
\leq & A^{p} +  C \eps^{-2}  [\eps^{-2}+T_\eps] a^{p}
+ C \eps^{-4} a A^{p-1}\\
\leq & C \eps^{-2}  [\eps^{-2}+T_\eps] a^{p}.
\end{split}
\end{equation*}
\end{proof}

We proceed now to the proof of the following main stability
result.

\begin{theorem}\label{newh1thm}
Consider the exit time
$$
\tau_\eps:=\inf\{t\in[0,T_\eps] \ : \ \|\nabla v(t)\| >
C_0\eps^{k-4-\tilde\kappa} \text{ or } \|v(t)\|>C_0\eps^{k-2}\} ,
$$
where $T_\eps:=\varepsilon^{-N}$ for arbitrary and fixed large
$N>0$ and for arbitrary and small $\tilde\kappa>0$. Let also,
$$k> 5 \;, \quad
\|v(0)\|< \eta\eps^{k-2} \quad \text{and}\quad \|\nabla v(0)\|<
\eta\eps^{k-4}.
$$
In addition, let
$$\eta_0\leq C\eps^{2k-2+\tilde{k}}
\quad \text{and}\quad \eta_2 \leq C\eps^{2k-6},
$$
be the assumptions for the noise, for $\tilde{k}>0$ small. Then
the probability $\mathbb{P} (\tau_\eps < T_\eps )$ is smaller than
any power of $\varepsilon$, as $\varepsilon \to 0$. So, for very
large times and with high probability, the solution stays close
to the manifold in the $H^1_\eps$-norm.
\end{theorem}

\begin{proof}
 Obviously, we have
 \[\mathbb{P}(\tau_\eps<T_\eps)
 \leq \mathbb{P}(\|\nabla v(\tau_\eps)\|\geq C_0\eps^{k-4-\tilde\kappa})
+\mathbb{P}(\|v(\tau_\eps)\|\geq C_0\eps^{k-2}).
 \]
 Using now Theorem \ref{newl2thm} for any large $\ell > 1$ and
 Chebychev's inequality, we obtain
\begin{equation*}
\begin{split}
\mathbb{P}(\tau_\eps<T_\eps)
&\leq C\eps^{-2p(k-4-\tilde\kappa)} \mathbb{E}\|\nabla v(\tau_\eps)\|^{2p}+C\eps^{\ell}
\\
&\leq C\eps^{2p\tilde\kappa}  \eps^{-2}  [\eps^{-2}+T_\eps]
+C\eps^{\ell},
\end{split}
\end{equation*}
where  Lemma \ref{lem1h3} was applied. Choosing $p \gg
1/2\tilde\kappa$ yields the result.
\end{proof}
Let us rephrase Theorem \ref{newh1thm} slightly:
\begin{corollary}
Under the assumptions of Theorem \ref{newh1thm} and if $w(0)\in
\mathcal{N}_{L^2}^{\eta \eps^{k-2}}$ and $\nabla w(0)\in
\mathcal{N}_{L^2}^{\eta \eps^{k-4}}$ for any $\eta
>0$, then for any sufficiently large $C>\eta$ and any
$q\in\mathbb{N}$ there exists a constant $C_q>0$ such that
\[
\mathbb{P} \Big( w(t)\in \mathcal{N}_{H^1_\eps}^{C
\eps^{k-3-\tilde{k}}} \text{ for all } t\in [0,\eps^{-q}] \Big)
\geq 1 - C_q\eps^q.
\]
\end{corollary}

%



\section{Estimates}
\label{sec44}

In this section, we will present the estimates of
$(\partial_\xi^2 u,\partial_\xi u)$ and of $\partial_\xi u$,
$\partial_\xi^2 u$, $\partial_\xi^3 u$ in various norms, used
throughout the previous sections.
Recall that $\Gamma$ was the small semicircle where $u^\xi=0$, and apart from a
small neighborhood, $u^\xi\approx 1$ inside and  $u^\xi\approx -1$ outside.
Our proof extends certain lower order results of
\cite{acfu} derived for the deterministic problem.  First, we estimate
the scalar product between $\partial_\xi^2 u$ and $\partial_\xi u$, which
can be bounded much better than via Cauchy-Schwarz.
\begin{theorem}\label{fest}
The following estimate holds true
\begin{equation}\label{estder}
(\partial_\xi^2 u,\partial_\xi u)=\mathcal{O}(\eps^{-1}).
\end{equation}
\end{theorem}
\begin{proof}
We may consider the case that $\Omega_\delta$ is a normal graph
over the unit sphere, so that any $x\in\Omega_\delta$ is
represented by $x=(r\cos(\theta),r\sin(\theta))$ for any
$0\leq\theta < 2\pi$ and any $0\leq r < R(\theta)$, where
$R(\theta)$ is the distance of the point of the boundary
$\partial\Omega_\delta$ from the origin, at the angle $\theta$.
This is not restrictive since, as we shall see, our integral
vanishes outside a neighborhood of a point on
$\partial\Omega_\delta$. The coordinate $r$ here should not be
confused with the local coordinate near $\Gamma$. Therefore, we
have
\begin{equation*}
(\partial_\xi^2 u,\partial_\xi u)=\int_{\Omega_\delta}
\partial_\xi^2 u\partial_\xi u
dx=\int_0^{2\pi}\int_0^{R(\theta)}\frac{1}{2}\frac{d}{d\xi}(\partial_\xi u)^2rdr\,d\theta.
\end{equation*}
Observe that if
$\partial\Omega_\delta=(a(\theta),b(\theta))=(R(\theta)\cos(\theta),R(\theta)\sin(\theta))$ for
$t\in[0,2\pi]$ then, the arc-length parameter $\xi$ of
$\partial\Omega_\delta$ is given by
$$\xi(\theta)=\int_0^\theta(a'(t)^2+b'(t)^2)^{1/2}dt.$$
 Therefore,
$\xi_\theta=(R'(\theta)^2+R(\theta)^2)^{1/2}$ and thus
$$d\xi=\xi_\theta d\theta=(R'(\theta)^2+R(\theta)^2)^{1/2}
d\theta.$$

 Setting $L:=|\partial\Omega_\delta|$ and using that the boundary is a
closed curve, it follows for $\hat{R}(\xi)=R(\theta)$ that
\begin{equation*}
\begin{split}
(\partial_\xi^2 u,\partial_\xi
(u))=&\int_0^{L}\int_0^{\hat{R}(\xi)}\frac{1}{2}\frac{d}{d\xi}(\partial_\xi
(u))^2rdr
(R'(\theta)^2+R(\theta)^2)^{-1/2}d\xi\\
=&\int_0^{L}
\frac{1}{2}\frac{d}{d\xi}\Big{(}\int_0^{\hat{R}(\xi)}(\partial_\xi u)^2rdr\Big{)}(R'(\theta)^2+R(\theta)^2)^{-1/2}d\xi\\
&- \int_0^{L}\frac{1}{2}\hat{R}_\xi(\xi)(\partial_\xi
(u)(\hat{R},\xi))^2\hat{R}(\xi)
(R'(\theta)^2+R(\theta)^2)^{-1/2}d\xi\\
=&\Big{[}
\frac{1}{2}\Big{(}\int_0^{\hat{R}(\xi)}(\partial_\xi u)^2rdr\Big{)}(R'(\theta)^2+R(\theta)^2)^{-1/2}\Big{]}_0^L\\
&-\int_0^{L}
\frac{1}{2}\Big{(}\int_0^{\hat{R}(\xi)}(\partial_\xi u)^2rdr\Big{)} \partial_\xi \Big{(}(R'(\theta)^2+R(\theta)^2)^{-1/2}\Big{)} d\xi\\
&- \int_0^{L}\frac{1}{2}\hat{R}_\xi(\xi)(\partial_\xi
(u)(\hat{R},\xi))^2\hat{R}(\xi)
(R'(\theta)^2+R(\theta)^2)^{-1/2}d\xi\\
=&0-\int_0^{L}
\frac{1}{2}\Big{(}\int_0^{\hat{R}(\xi)}(\partial_\xi u)^2rdr\Big{)}\partial_\xi\Big{(}(R'(\theta)^2+R(\theta)^2)^{-1/2}\Big{)} d\xi\\
&- \int_0^{L}\frac{1}{2}\hat{R}_\xi(\xi)(\partial_\xi
(u)(\hat{R},\xi))^2\hat{R}(\xi)
(R'(\theta)^2+R(\theta)^2)^{-1/2}d\xi.
\end{split}
\end{equation*}
Note that the construction of $u$ in \cite{acfu} shows that $\partial_\xi u$
vanishes outside a neighborhood of $\Gamma$ of width $2\eps \text{log}^2 \eps$,
which allows us to use the representation of $\partial\Omega_\delta$.

We will use the notation $M \sim \mathcal{O} (\delta^s)$ to mean
that there are constants $C_1, C_2 >0$ such that $C_1\delta^s\leq
M\leq C_2\delta^s$ for $\delta$ sufficiently small.

Returning to the original set
$\Omega=\delta\Omega_\delta$ with arc-length parameter for its
boundary $\tilde{\xi}=\delta \xi$.
Moreover, we frequently use indices to denote derivatives to stay close to the notation of \cite{acfu}.
Then for the distance, $\tilde{R}(\tilde{\xi})$, of the boundary, $\partial \Omega$, from the
origin,  we have
$$\tilde{R}(\tilde{\xi})=\delta\hat{R}(\xi) \sim \mathcal{O}(1),\;\;
\tilde{R}_{\tilde{\xi}}(\tilde{\xi}) \sim \mathcal{O}(1),\;\;
\tilde{R}_{\tilde{\xi}\tilde{\xi}}(\tilde{\xi}) \sim \mathcal{O}(1).$$
So
\begin{equation*}
\begin{split}
&\hat{R}(\xi)\sim\mathcal{O}(\delta^{-1}),\\
&\hat{R}_\xi(\xi)\sim\mathcal{O}(\delta^{-1})\tilde{R}_{\tilde{\xi}}(\tilde{\xi})\tilde{\xi}_\xi\sim
\mathcal{O}(\delta^{-1})\mathcal{O}(1)\delta\sim\mathcal{O}(1),\\
&\hat{R}_{\xi\xi}(\xi)\sim\mathcal{O}(\delta^{-1})\tilde{R}_{\tilde{\xi}\tilde{\xi}}(\tilde{\xi})\tilde{\xi}_\xi^2+
\mathcal{O}(\delta^{-1})\tilde{R}_{\tilde{\xi}}(\tilde{\xi})\tilde{\xi}_{\xi\xi}\sim
\mathcal{O}(\delta^{-1})\mathcal{O}(1)\delta^2\sim\mathcal{O}(\delta).
\end{split}
\end{equation*}
If  $\bar{R}(\theta)$ is the distance from the origin to the
boundary $\partial\Omega$ at angle $\theta$, we have
\begin{equation*}
\begin{split}
&\hat{R}(\xi)=R(\theta),\\
&\delta\hat{R}(\xi)=\tilde{R}(\tilde{\xi})=\bar{R}(\theta)\sim\mathcal{O}(1),\\
&\bar{R}_\theta(\theta)\sim\mathcal{O}(1),\\
&\bar{R}_{\theta\theta}(\theta)\sim\mathcal{O}(1),
\end{split}
\end{equation*}
and thus
\begin{equation*}
\begin{split}
&\bar{R}(\theta)=\delta R(\theta)\sim\mathcal{O}(1),\\
&R(\theta)\sim\mathcal{O}(\delta^{-1}),\\
&R_\theta(\theta)=\delta^{-1}\bar{R}_\theta(\theta)\sim\mathcal{O}(\delta^{-1}),\\
&R_{\theta\theta}(\theta)=\delta^{-1}\bar{R}_{\theta\theta}(\theta)\sim\mathcal{O}(\delta^{-1}).
\end{split}
\end{equation*}
But $R(\theta)=\hat{R}(\xi)$ so
$R_\theta(\theta)=\hat{R}_{\xi}(\xi)\xi_\theta=\hat{R}_{\xi}(\xi)(R_\theta^2+R^2)^{1/2}$.
So
$$
\partial_\xi\Big{(}(R'(\theta)^2+R(\theta)^2)^{-1/2}\Big{)}
=\partial_\xi \Big{(}\frac{\hat{R}_\xi(\xi)}{R_\theta(\theta)}\Big{)}
= \partial_\xi \frac{\hat{R}_{\xi\xi}R_\theta-\hat{R}
R_{\theta\theta}\theta_\xi}{R_\theta^2}.$$
In $\Omega$ we have
$$\frac{\partial\theta}{\partial\tilde{\xi}}\sim \mathcal{O}(1)\,
\quad\text{and}\quad
\tilde{\xi}=\delta\xi\;.$$
Thus
$$\theta_\xi\sim\mathcal{O}(\delta).
$$
So we get
$$\partial_\xi \Big{(}(R'(\theta)^2+R(\theta)^2)^{-1/2}\Big{)}
\sim\frac{\mathcal{O}(\delta)\mathcal{O}(\delta^{-1})-\mathcal{O}(1)
\mathcal{O}(\delta^{-1})\mathcal{O}(\delta)}{\delta^{-2}}\sim\mathcal{O}(\delta^2).
$$
But (cf.\ \cite{acfu}, p. 294) in a neighborhood of $\Gamma$ of width $\eps \text{log}^2\eps$,
using local coordinates $(r,s)$ with $r$ being signed distance from $\Gamma$ and $s$ being arclength along $\Gamma$,
\begin{equation}
 \label{e:dxiuso}
 \partial_\xi u=-\frac{1}{\eps}\Big{\{}\cos(\pi
s/|\Gamma|)\dot{U}(\frac r\eps)+\mathcal{O}(\eps)\Big{\}}=\mathcal{O}(\eps^{-1}),
\end{equation}
where $U$ is the heteroclinic solution to the one-dimensional problem
connecting $\pm 1$, given by (\ref{heterocl}).

The smooth cut-off function maintains this estimate on the support of $\partial_\xi u$,
being a neighborhood of $\Gamma$ of width $2\eps\text{log}^2\eps$.
Furthermore, $\dot{U}$ decays exponentially, as shown in the classical work of Fife and
McLeod \cite{fmcl}.

 So we obtain
\begin{equation*}
\begin{split}
|(\partial_\xi^2 u,\partial_\xi u)|=&\int_0^{L}
\frac{1}{2}\Big{(}\int_0^{\hat{R}(\xi)}|\partial_\xi
(u)|^2rdr\Big{)}|\Big{(}(R'(\theta)^2+R(\theta)^2)^{-1/2}\Big{)}_\xi|
d\xi\\
&+ \int_0^{L}\frac{1}{2}|\hat{R}_\xi(\xi)||\partial_\xi
(u)(\hat{R},\xi)|^2\hat{R}(\xi)
(R'(\theta)^2+R(\theta)^2)^{-1/2}d\xi\\
=&\mathcal{O}(\eps^{-1})\mathcal{O}(\delta^2)
+\mathcal{O}(\eps^{-1})+\text{ higher order terms }\\
=&\mathcal{O}(\eps^{-1}).
\end{split}
\end{equation*}
\end{proof}
We now give estimates for various derivatives of $u$ with respect to $\xi$.
Some of these are already given in \cite{acfu} but are repeated here for completeness of presentation.
\begin{theorem}\label{mainest}
It holds that
\begin{equation}\label{eex}
\|\partial_\xi u\|_{L^1}=\mathcal{O}(1),
\qquad\qquad
\|\partial_\xi u\|= \mathcal{O}(\eps^{-\frac{1}{2}}),
\end{equation}
\begin{equation}\label{estx}
\|\partial_\xi^2 u\|=\mathcal{O}(\eps^{-\frac{3}{2}}),
\quad\text{and}\quad
\|\partial_\xi^3 u\|= \mathcal{O}(\eps^{-\frac{5}{2}}).
\end{equation}
\end{theorem}
\begin{proof}
Following \cite{acfu}, p. 294, in a neighborhood of $\Gamma$ of
width $\eps\text{log}^2\eps$,
\begin{equation}\label{e2}
\partial_\xi u=-\frac{1}{\eps}\Big{\{}\cos(\pi
s/|\Gamma|)\dot{U}(R)+\mathcal{O}(\eps)\Big{\}}.
\end{equation}
where 
$R =\frac{r}{\eps}$, and $(r,s)$ are local coordinates in this neighborhood, with
 $r$  being signed distance and $s$ being arc length along $\Gamma.$ Recall that $U$ is defined as
the heteroclinic solution  to \eqref{heterocl}. See also \eqref{e:dxiuso}.

The function $u$ is the sum of two terms written in local
coordinates, an interior expansion $u^I$ and a
corner layer expansion $u^{B\pm}$:
$$u=u^I+u^{B},$$
where (from \cite{acfu} pp. 251, 254)
\begin{equation}\label{m1}
u^I=U(R)+\eps\sum_{j\geq 0}\eps^ju_j^I(R,\cdots),
\end{equation}
and
\begin{equation}\label{m2}
u^{B}=\sum_{j\geq 1}\eps^ju_j^{B\pm}(R,\cdots).
\end{equation}
Let $D_a^k$ denote the $k$ partial derivative with respect to variable $a$. For
the interior layer expansion it holds that (see \cite{acfu} p. 252)
\begin{equation}\label{m3}
|D_R^mD_s^nD_\xi^lu_j^I|= \mathcal{O}(1),
\end{equation}
for any integer $m,n,l\geq 0$. Since the interior expansion in
local coordinates has a smooth extension to the whole domain (by
\eqref{m1}, this means that $u_j^I$ have smooth extensions to the
whole domain), then
\eqref{m3}, which is true for $u_j^I$ in local coordinates, is true
for their smooth extensions too. On the other hand for a given
$\xi$, the construction of $u^{B\pm}$ in local coordinates (see
\cite{acfu} p. 241) permits as to derive  that
$D_R^mD_s^nD_\xi^lu_j^{B\pm}$ are bounded uniformly in $\eps$, i.e.,
\begin{equation}\label{m4}
|D_R^mD_s^nD_\xi^lu_j^{B\pm}| = \mathcal{O}(1),
\end{equation}
for any integer $m,n,l\geq 0$.

Furthermore, $u$ is smoothly extended to the whole domain, cf.\ p. 257
of \cite{acfu}, and thus, is given as
$$u=\tilde{u^I}+\tilde{u^B},$$ for $\tilde{u^I}$ and $\tilde{u^B}$ suitable modifications of
$u^I$ and $u^B$ respectively. More specifically, for the cut-off function $\zeta\in
C^\infty$ with $\zeta(s)=1$ if $s>1$, $\zeta(s)=0$ if $s<0$ and
$s\zeta'(s)>0$ on $\mathbb{R}$, $\tilde{u^I}$ is given as
\begin{equation}\label{m1*}
\begin{split}
\tilde{u^I}:=&(1-\zeta^{+}-\zeta^{-})(U(R)+\eps\sum_{j\geq
0}\eps^ju_j^I(R,\cdots))+\zeta^{+}u^{+}(\xi)+\zeta^{-}u^{-}(\xi)\\
=&(1-\zeta^{+}-\zeta^{-})u^I+\zeta^{+}u^{+}(\xi)+\zeta^{-}u^{-}(\xi),
\end{split}
\end{equation}
where
$$\zeta^{\pm}=\zeta(\pm\tfrac{r(\xi,\cdots)}{\eps \ln\eps^2}-1),
\quad\text{and}\quad
u^{\pm}=\pm 1+\eps\sum_{j\geq 0}\eps^ju_j^{\pm}.
$$
Note that $u_j^{\pm}$ are smooth and uniformly bounded in $\eps$,
and the same holds for their derivatives of any order. Easily, by
taking the derivative in $\xi$ we obtain
\begin{equation*}
\begin{split}
\partial_\xi(\tilde{u^I})=&\partial_\xi(1-\zeta^{+}-\zeta^{-})u^I
+(1-\zeta^{+}-\zeta^{-})\partial_\xi(u^I)+
\partial_\xi(\zeta^{+})u^{+}+\zeta^{+}\partial_\xi(u^{+})
+\partial_\xi(\zeta^{-})u^{-}+\zeta^{-}\partial_\xi(u^{-})\\
=&\frac{1}{\eps\ln\eps^2}\mathcal{O}(u^I)+\frac{1}{\eps}\mathcal{O}(u^I)
+\frac{1}{\eps\ln\eps^2}\mathcal{O}(1)+\mathcal{O}(u^I)\sim\mathcal{O}(\partial_\xi(u^I)),
\end{split}
\end{equation*}
where we used that
$\frac{1}{\eps}u^I = \mathcal{O}(\partial_\xi u^I)$ (to be proved
in the sequel).
So, the order of $\partial_\xi(\tilde{u^I})$ is
this of $\partial_\xi(u^I)$, while by induction the same is
happening for the derivatives of higher order i.e.,\
$\partial_\xi^m(\tilde{u^I})\sim\mathcal{O}(\partial_\xi^m(u^I))$
for $m=1,2,3$. In the sequel, we compute
$\mathcal{O}(\partial_\xi^m(u^I))$ in detail.

An analogous construction for $\tilde{u^B}$ will give the same
result, which means that
$\partial_\xi^m(\tilde{u^B})\sim\mathcal{O}(\partial_\xi^m(u^B))$,
so it is sufficient if we estimate $\partial_\xi^m(u^B)$ for $m=1,2,3$.

In addition, $D_\xi^ls$, $D_\xi^lr$ exist for  $l\geq 0$,
and is uniformly bounded for any $\eps$ (for convenience
cf.\ in \cite{acfu} p. 294, the formulas $r_\xi=-\cos(\pi
s/|\Gamma|)+\mathcal{O}(\delta)$ and
$s_\xi=\mathcal{O}(\eps+h+|r|)$) so,
\begin{equation}\label{m5}
|D_\xi^ls|,\;\;|D_\xi^lr|< C,
\end{equation}
for any $l\geq 0$ uniformly in $\eps$.

But, cf.\ \cite{acfu} p. 258
$$\frac{\partial x}{\partial(R,s)}=\eps(1+\eps
R\mathcal{K}),$$ thus, for any $g=g(R,\cdots)$ of compact support,
it holds that (as in p. 258 of \cite{acfu})
\begin{equation*}
\int_{\Omega_\delta}g(R,\cdots)dx=\eps\int_{-\infty}^{+\infty}dR
\int_{S^-(R,\xi,\eps)}^{S^+(R,\xi,\eps)}(1+\eps
R\mathcal{K})g(R,\cdots)ds,
\end{equation*}
for the resulting $S^-$ and $S^+$, where the change of variables
from local to global coordinates is valid. Since $\Omega_\delta$
is bounded and $R=\frac{r}{\eps}$ then if $g(R,\cdots)$ is bounded
uniformly for any $\eps$ for any $R\in
\mathbb{R}$ then 
\begin{equation}\label{gencomp}
\int_{\Omega_\delta}g(R,\cdots)dx=\mathcal{O}(1).
\end{equation}

Using the above considerations then we can derive the estimates
of $\|\partial_\xi u\|_{L^1}$, $\|\partial_\xi u\|$,
$\|\partial_\xi^2 u\|$ and $\|\partial_\xi^3 u\|$, where
$\partial_\xi u=\frac{\partial}{\partial\xi}(u(r,s,\xi,\eps))$
and so on.

By using \eqref{e2} we obtain
\begin{equation*}
\begin{split}
\int_{\Omega_\delta}|\partial_\xi
(u)|dx=&\mathcal{O}\Big{(}\eps\int_{-\infty}^{+\infty}dR
\int_{S^-(R,\xi,\eps)}^{S^+(R,\xi,\eps)}(1+\eps R\mathcal{K})|\partial_\xi u|ds\Big{)}\\
=&\mathcal{O}\Big{(}\eps\int_{-\infty}^{+\infty}dR
\int_{S^-(R,\xi,\eps)}^{S^+(R,\xi,\eps)}(1+\eps
R\mathcal{K})\frac{1}{\eps}|\cos(\pi
s/|\Gamma|)||\dot{U}(R)|ds\Big{)}\\
&\qquad
+ \mathcal{O}\Big{(}\eps\int_{a}^{b}dR
\int_{S^-(R,\xi,\eps)}^{S^+(R,\xi,\eps)}(1+\eps R\mathcal{K})
\frac{1}{\eps}\mathcal{O}(\eps)ds\Big{)}\\
=&\mathcal{O}(1),
\end{split}
\end{equation*}
i.e.,\ we proved the left equation in \eqref{eex}. Here, we used that
$\int_{-\infty}^{+\infty}|S^+(R,\xi,\eps)-S^-(R,\xi,\eps)||\dot{U}(R)|dR$
is bounded uniformly in $\eps$ and that the order for
$\partial_\xi u$ is given by the order of the local
representation $\partial_\xi(u^I+u^B)$, estimated by \eqref{e2}.

Further it follows that
\begin{equation*}
\begin{split}
\int_{\Omega_\delta}|\partial_\xi
(u)|^2dx&=\mathcal{O}\Big{(}\eps\int_{-\infty}^{+\infty}dR
\int_{S^-(R,\xi,\eps)}^{S^+(R,\xi,\eps)}(1+\eps R\mathcal{K})|\partial_\xi u|^2ds\Big{)}\\
&=\mathcal{O}\Big{(}\eps\int_{-\infty}^{+\infty}dR
\int_{S^-(R,\xi,\eps)}^{S^+(R,\xi,\eps)}(1+\eps
R\mathcal{K})\frac{1}{\eps^2}|\cos(\pi
s/|\Gamma|)|^2|\dot{U}(R)|^2ds+C\Big{)}=\mathcal{O}(\eps^{-1}),
\end{split}
\end{equation*}
so, we get the right equality in \eqref{eex}. Here, we used that
$\int_{-\infty}^{+\infty}|S^+(R,\xi,\eps)-S^-(R,\xi,\eps)||\dot{U}(R)|^2dR$
is bounded uniformly in $\eps$.

Obviously, since
$$
\partial_\xi u=u_rr_\xi+u_ss_\xi+u_\xi,
$$
it holds that
\begin{equation*}
\begin{split}
\partial_\xi^2 u=&(u_{rr}r_\xi+u_{rs}s_\xi+u_{r\xi}\xi_\xi)r_\xi
+u_rr_{\xi\xi}\\
&+(u_{sr}r_\xi+u_{ss}s_\xi+u_{s\xi}\xi_\xi)s_\xi
+u_ss_{\xi\xi}\\
&+u_{\xi r}r_\xi+u_{\xi s}s_\xi+u_{\xi\xi}\xi_\xi.
\end{split}
\end{equation*}
Hence, the worst term is $u_{rr}$ multiplied by a uniformly (in $\eps$)
bounded  quantity. More specifically, using \eqref{m1},
\eqref{m2}, \eqref{m3}, \eqref{m4} and \eqref{m5}, we observe
that any derivative in $r$ gives $\eps^{-1}$ since
$R=\frac{r}{\eps}$, while all the other derivatives of $s,r$ are
uniformly bounded in $\eps$ and the same is true for the
derivatives of $u$ in $s,\xi$. Further, considering the term
$\partial^3_\xi(u)$, the chain rule analogously gives that the
worst order is given by $u_{rrr}$.

In details, \eqref{m1} and \eqref{m3} give for $m,n,l\geq 0$
\begin{equation}\label{m6}
\begin{split}
D_r^mD_s^nD_\xi^lu^I&=D_rU(R)+\eps\eps^{-m}\sum_{j\geq
0}\eps^jD_R^mD_s^nD_\xi^lu_j^I\\
&\leq \eps^{-m}D_R^mU(R)+C\eps\eps^{-m},
\end{split}
\end{equation}
while by \eqref{m2} and \eqref{m4} we obtain
\begin{equation}\label{m7}
\begin{split}
D_r^mD_s^nD_\xi^lu^B&=\eps^{-m}\sum_{j\geq
1}\eps^jD_R^mD_s^nD_\xi^lu_j^{B\pm}\\
&\leq C\eps^{-m}\sum_{j\geq 1}\eps^j\leq C\eps^{-m}\eps.
\end{split}
\end{equation}
Since $u_{rr}=u_{rr}^I+u_{rr}^B$, then taking $m=2$ we get using
\eqref{m6} and \eqref{m7}, and the regularity of the second
derivative of the heteroclinic,
\begin{equation*}
\begin{split}
\int_{\Omega_\delta}|\partial_\xi^2 u|^2dx&\leq
C\eps\int_{-\infty}^{+\infty}dR\int_{S^-}^{S^+}(1+\eps
R\mathcal{K})|u_{rr}^I|^2ds+C\eps\int_{-\infty}^{+\infty}dR\int_{S^-}^{S^+}(1+\eps
R\mathcal{K})|u_{rr}^B|^2ds\\
&\leq C\eps\int_{-\infty}^{+\infty}\eps^{-2\cdot
2}dR\int_{S^-}^{S^+}(1+\eps
R\mathcal{K})|D_R^2U(R)|^2ds+C\eps^2\eps^{-2\cdot
2}+C\eps^{-2\cdot 2}\eps^2
\\&
\leq C\eps^{-3},
\end{split}
\end{equation*}
and thus, in accordance to \cite{acfu} p. 297, we proved the left statement in
\eqref{estx}, i.e.,\
\begin{equation*}
\|\partial_\xi^2 u\|=\mathcal{O}(\eps^{-\frac{3}{2}}).
\end{equation*}
For the previous relation we used that
$\int_{-\infty}^{+\infty}|S^+(R,\xi,\eps)-S^-(R,\xi,\eps)||D_R^2U(R)|^2dR$
is bounded uniformly in $\eps$.

The analogous computation by taking $m=3$ in \eqref{m6} and
\eqref{m7}, since the third derivative of the heteroclinic is
regular, yields
\begin{equation*}
\begin{split}
\int_{\Omega_\delta}|\partial_\xi^3 u|^2dx&\leq
C\eps\int_{-\infty}^{+\infty}dR\int_{S^-}^{S^+}(1+\eps
R\mathcal{K})|u_{rrr}^I|^2ds+C\eps\int_{-\infty}^{+\infty}dR\int_{S^-}^{S^+}(1+\eps
R\mathcal{K})|u_{rrr}^B|^2ds\\
&\leq C\eps\int_{-\infty}^{+\infty}\eps^{-2\cdot 3}dR
\int_{S^-}^{S^+}(1+\eps
R\mathcal{K})|D_R^3U(R)|^2ds+C\eps^2\eps^{-2\cdot
3}+C\eps^{-2\cdot 3}\eps^2\leq C\eps^{-5}.
\end{split}
\end{equation*}
Thus, we proved the right inequality in \eqref{estx}, i.e., that
\begin{equation*}
\|\partial_\xi^3 u\|= \mathcal{O}(\eps^{-\frac{5}{2}}).
\end{equation*}
Here, we used that
$\int_{-\infty}^{+\infty}|S^+(R,\xi,\eps)-S^-(R,\xi,\eps)||D_R^3U(R)|^2dR$
is bounded uniformly in $\eps$.
\end{proof}
\begin{remark}
We note that by  \cite[p. 297]{acfu}  also
$\|\partial_\xi u\|\geq C\eps^{-\frac{1}{2}}.$ Thus
there exist $0<C_1\leq C_2$ such that
\[
C_1\eps^{-\frac{1}{2}}\leq
\|\partial_\xi u\|\leq C_2\eps^{-\frac{1}{2}}.
\]
\end{remark}
\begin{remark}
We used regularity for the heteroclinic $U(R)$ up to derivatives of third
order, see \cite{fmcl}.
\end{remark}
\begin{remark}
From the proof of Theorem \ref{mainest} it is obvious that for
any $k\geq 1$
$$
\|\partial_x^k u\|\sim \|\partial_\xi^k(u)\|,
\quad\text{and}\quad
\|\partial_x^k u\|_{\infty}\sim \|\partial_\xi^k(u)\|_\infty.
$$
Furthermore, by
\cite{acfu}, it holds that
$$\|\partial_\xi u\|_{\infty}\leq C\eps^{-1},
\quad\text{and thus}
\|\nabla u\|_{\infty}\leq C\eps^{-1}.
$$
\end{remark}

\end{document}